\documentclass[11pt]{amsart}
\usepackage[top=1.3in, left=1.3in, right=1.3in, bottom=1.3in]{geometry}

\usepackage{amsmath,amsfonts,amsthm,graphicx}
\usepackage{amsmath}
\usepackage{hyperref}
\usepackage{graphicx,psfrag}  
\usepackage[psamsfonts]{amssymb}
\usepackage{amscd}
\usepackage[all,cmtip]{xy}
\title[Rates of mixing]{Rates of mixing for the Weil-Petersson geodesic flow II: exponential mixing in  exceptional moduli spaces}

\author{K. Burns, H. Masur, C. Matheus and  A. Wilkinson}
\thanks{K.B. was supported by NSF grant DMS-1001959, H.M. was supported by NSF grant DMS 1205016, C.M. was supported by ANR grant ``GeoDyM'' (ANR-11-BS01-0004), and  A.W. was supported by NSF grant DMS-1316534.}
\address{Keith Burns: Department of Mathematics, Northwestern University, 2033
Sheridan Road, Evanston, IL 60208-2730 USA.}
\email{burns@math.northwestern.edu.}
\urladdr{http://www.math.northwestern.edu/~burns/}
\address{Howard Masur: Department of Mathematics, University of Chicago, 5734 S.
University, Chicago, IL 60637, USA.}
\email{masur@math.uchicago.edu.}
\urladdr{http://www.math.uchicago.edu/~masur/}
\address{Carlos Matheus: Universit\'e Paris 13, Sorbonne Paris Cit\'e, LAGA, CNRS (UMR 7539), F-93439, Villetaneuse, France.}
\email{matheus@impa.br.}
\urladdr{http://www.impa.br/$\sim$cmateus}
\address{Amie Wilkinson: Department of Mathematics, University of Chicago, 5734 S.
University, Chicago, IL 60637, USA.}
\email{wilkinso@math.uchicago.edu.}
\urladdr{http://math.uchicago.edu/~wilkinso/}

\date{\today}
\theoremstyle{plain}

\newtheorem{mainthm}{Theorem}
\newtheorem{theorem}{Theorem}[section]
\newtheorem{proposition}[theorem]{Proposition}
\newtheorem{lemma}[theorem]{Lemma}
\newtheorem{corollary}[theorem]{Corollary}

\def\title{\em}

\def\bar{\overline}

\def\cW{\mathcal{W}}

\def\cB{\mathcal{B}}

\def\cL{\mathcal{L}}
\def\cA{\mathcal{A}}
\def\cB{\mathcal{B}}
\def\cI{\mathcal{I}}
\def\cJ{\mathcal{J}}

\def\cG{\mathcal{G}}

\def\cC{\mathcal{C}}

\def\cT{\mathcal{T}}
\def\B{\mathcal{B}}
\def\cR{\mathcal{R}}
\def\cD{\mathcal{D}}

\def\cH{\mathcal{H}}

\def\cN{\mathcal{N}}
\def\cV{\mathcal{V}}
\def\cM{\mathcal{M}}

\def\transverse{\,\raise2pt\hbox to1em{\hfil$\top$\hfil}\hskip -1em \hbox
to1em{\hfil$\cap$\hfil}\,} 
\newcommand\vol{\operatorname{vol}}

\newcommand\RR{{\mathbb R}}

\newcommand\DD{{\mathbb D}}

\newcommand\ZZ{{\mathbb Z}}
 
\newlength{\figboxwidth} 

\begin{document}

\begin{abstract}{We establish exponential mixing for the geodesic flow $\varphi_t\colon T^1S\to T^1S$ of an incomplete, negatively curved surface $S$
with cusp-like singularities of a prescribed order.  As a consequence, we obtain that the Weil-Petersson flows for the moduli spaces $\cM_{1,1}$ and $\cM_{0,4}$ are exponentially mixing, in sharp contrast to the flows for $\cM_{g,n}$ with $3g-3+n>1$, which fail to be rapidly mixing.  In the proof, we present a new method of analyzing invariant foliations for hyperbolic flows with singularities, based on changing the Riemannian metric on the phase space $T^1S$ and rescaling the flow $\varphi_t$. }
\end{abstract}

\maketitle

\tableofcontents

\section*{Introduction}

Let $S$ be an oriented surface with finitely many punctures.  Suppose that $S$ is endowed with a negatively curved Riemannian metric and that in a neighborhood of each puncture the metric is ``asymptotically modeled" on a surface of revolution obtained by rotating the curve $y=  x^r$, for some $ r>2$, about the $x$-axis in $\RR^3$  (where $r$ may depend on the puncture).  The results in this paper allow us to conclude that the geodesic flow on $T^1S$ mixes exponentially fast.

Before stating the hypotheses precisely, we recall some facts about the metric on a surface $R$ of revolution for the function $y= x^r$.   This surface is negatively curved, incomplete and the curvature can be expressed as a function of the distance to the cusp point $p_0$ where $x=y=0$.  Denote by $\rho(\cdot,\cdot)$ the induced Riemannian path metric and $\delta \colon R\to \RR_{\geq 0}$ the Riemannian distance to the cusp:
$$
\delta (p) = \rho(p,p_0).
$$
Then for $r>1$, the Gaussian curvature on $R$ has the following asymptotic expansion in $\delta$, as $\delta\to 0$:
 \[K(p) = -\frac{r(r-1)}{\delta(p)^2} + O(\delta(p)^{-1}).\]
Our main theorem applies to any incomplete, negatively curved surface with
singularities of this form.  More precisely, we have:

\begin{mainthm}\label{t=main} Let $X$ be a closed surface, and let $\{p_1,\ldots, p_k\}\subset X$.  Suppose that the punctured surface $S = X\setminus \{p_1,\ldots, p_k\}$ carries a $C^5$, negatively curved Riemannian metric that extends to a complete distance metric $\rho$ on $X$.   
Assume that the lift of this metric to the universal cover $\widetilde S$ is geodesically convex.
Denote by $\delta_i\colon S\to \RR_+$ the distance
$\delta_i(p) = \rho(p, p_i)$, for $i=1,\ldots, k$.   

Assume that  there exist $r_1,\ldots,r_k > 2$ such that the Gaussian curvature $K$ satisfies
\[ K(p) = \sum_{i=1}^k -\frac{r_i(r_i-1)}{\delta_i(p)^2}  + O(\delta_i(p)^{-1}) \] 
and 
\[ \|\nabla^j  K(p)\| =   \sum_{i=1}^k O({\delta_i(p)^{-2-j}}),
\]
for $j=1,2,3$ and all $p\in S$. 

Then the geodesic flow $\varphi_t\colon T^1 S\to T^1S$ is exponentially mixing:  there exist constants  $c, C > 0$ such that
for every pair of  $C^1$ functions $u_1, u_2 \in  L^\infty(T^1S, \vol)$, we have
\[\left| \int_{T^1S} u_1 \, u_2\circ\varphi_t  \,d\vol - \int  u_1\, d\vol  \int  u_2\, d\vol   \right| \leq Ce^{-ct}\|u_1\|_{C^1} \|u_2\|_{C^1},
\]
for all $t>0$, where $\vol$ denotes the Riemannian volume on $T^1S$  (which is finite)  normalized so that 
$\vol(T^1S)=1$.
\end{mainthm}

The regularity hypotheses on $u_1, u_2$ are not optimal.  See Corollary~\ref{c=main} in the last section for precise formulations.

Theorem~\ref{t=main} has a direct application to the dynamics of the Weil-Petersson flow,
which is the geodesic flow for the Weil-Petersson metric $\langle\cdot,\cdot\rangle_{WP}$ of the moduli spaces  $\cM_{g,n}$ of Riemann surfaces of genus $g \geq  0$
and $n \geq 0$ punctures, defined for $3g - 3 + n \geq 1$.   For a discussion of the WP metric and properties 
of its flow, see the recent, related work \cite{BMMW1}.
As a corollary, we obtain the following result, which originally motivated this study.

\begin{corollary}\label{c=WP}  The Weil-Petersson geodesic flow on $T^1\cM_{(g,n)}$ mixes exponentially fast when $(g,n) = (1,1)$ or $(0,4)$.
\end{corollary}

\begin{proof}[Proof of Corollary]  Wolpert shows in \cite{Wol11} that the hypotheses of Theorem~\ref{t=main} are satisfied by the WP metric on $\cM_{g,n}$, for $3g - 3 + n = 1$.
\end{proof}

Mixing of the WP flow (for all $(g,n)$) had previously been established in \cite{BMW}.  For $(g,n) \notin \{ (1,1), (0,4)\}$, the conclusions of Corollary~\ref{c=WP} do {\em not} hold \cite{BMMW1}: for every $k>0$, there exist 
compactly supported, 
$C^k$ test functions $u_1, u_2$ such that the correlation between $u_1$ and $u_2\circ\varphi_T$ decays at best polynomially in $T$.

\begin{remark} The geodesic convexity assumption in Theorem~\ref{t=main} can be replaced by a variety of other equivalent assumptions.  For example, it is enough to assume that 
$\delta_i =  \beta_i + o(\delta_i)$, where $\beta_i$ is a convex  function (as is the case in the WP metric).  Alternatively, one may assume a more detailed expansion for the metric in the neighborhood of the cusps.  For example, the assumptions near the cusp are satisfied for a surface of revolution for the function $y= u(x)x^r$, where $u\colon [0,1]\to \RR_{\geq 0}$ is $C^5$, with $u(0)\neq 0$ and $r>2$.  One can easily formulate further perturbations of this metric outside the class of surfaces of revolutions for which the hypotheses of Theorem~\ref{t=main} hold near $\delta=0$.
\end{remark}

To simplify the exposition and reduce as much as possible the use of unspecified constants, we will assume
in our proof that $k=1$, so that $ S$ has only one cusp. 

\subsection{Discussion}  
In a landmark paper \cite{D}, Dolgopyat  established that the geodesic flow for any negatively-curved compact surface is exponentially mixing.  His techniques, building in part on earlier work of Ruelle, Pollicott and Chernov, have since been extracted and generalized in a series of papers, first by Baladi-Vall\'ee \cite{BV}, then Avila-Gou\"ezel-Yoccoz \cite{AGY}, and most recently in the work of Ara\'ujo-Melbourne \cite{AM}, upon which this paper relies.  

Ultimately, the obstructions to applying Dolgopyat's original argument in this context are purely technical, but to overcome these obstructions in any context is the heart of the matter.  The solution to the analogous problem in the billiards context -- exponential mixing for Sinai billiards of finite horizon -- has only been recently established \cite{BDL}.

To prove exponential mixing using the symbolic-dynamical approach of Dolgopyat, Baladi-Vall\'ee $\hbox{et. al.}$, one  constructs a section to the flow
with certain analytic and symbolic dynamical properties.  In sum, one seeks a surface $\Sigma\subset T^1S$ transverse to the flow $\varphi_t$ in the three manifold $T^1S$ on which the dynamics of the return map can be tightly organized. 

  In particular, we seek a return time function $R\colon \Sigma_0\to \RR_{>0}$ 
defined on a full measure subset $\Sigma_0\subset \Sigma$, with
$\varphi_{R(v)}(v)\in \Sigma$ for all $v\in \Sigma_0$ and so that the dynamics of $F\colon v\mapsto \varphi_{R(v)}(v)$ on $\Sigma_0$ are hyperbolic and can be modeled on a  full shift on countably many symbols.  For $\varphi_t$ to be exponentially mixing, the function $R$ must be constant along stable manifolds, have
exponential tails and satisfy a non-integrability condition (UNI) (which will hold automatically if the flow $\varphi_t$ preserves a contact form, as is the case here). 

 Whereas in \cite{BV} and \cite{AGY} the map $F$ is required to be piecewise uniformly $C^2$, the regularity of $F$ is relaxed to $C^{1+\alpha}$ in \cite{AM}.  This relaxation in regularity might seem mild, but it is crucial in applications to nonuniformly hyperbolic flows with singularities.  The reason is that the surface $\Sigma$ is required to be saturated
by leaves of the (strong) stable foliation $\cW^s$ for the flow $\varphi_t$.   The smoothness of the foliation $\cW^s$ thus dictates the smoothness of the surface $\Sigma$ which then determines the smoothness of $F$ (up to the smoothness of the original flow $\varphi_t$).   {\em Even in the case of contact Anosov flows in dimension 3,} the foliation $\cW^s$ is no better than $C^{1+\alpha}$, for some $\alpha\in(0,1)$, unless the flow is algebraic in nature.\footnote{This issue is bypassed in the application to the Teichm\"uller flow in \cite{AG, AGY} because there the stable and unstable foliations are locally affine.}

While it has long been known that this $C^{1+\alpha}$ regularity condition holds for the stable and unstable foliations of contact Anosov flows in dimension 3, this is far from the case for singular and nonuniformly hyperbolic flows, even in low dimension.  In the context of this paper, the geodesic flow $\varphi_t$ is not even complete, and the standard singular hyperbolic theory fails to produce $\varphi_t$-invariant foliations $\cW^u$ and $\cW^s$, let alone foliations with $C^{1+\alpha}$ regularity.  

The flows $\varphi_t$ considered here, while incomplete, bear several resemblances to Anosov flows.  Most notably, there exist  $D\varphi_t$-invariant stable and unstable cone fields that are defined {\em everywhere} in $T^1S$.  The angle between these cone fields 
tends to zero as the basepoint in $T^1S$ approaches the singularity.  The action of $D\varphi_t$ in these cones 
is strongly hyperbolic, with the strength of the hyperbolicity approaching infinity as the orbit comes close to the 
singularity. 

The key observation in this paper is that by changing the Riemannian metric on $T^1S$ {\em and} performing a natural time change in  $\varphi_t$ one obtains a volume-preserving  {\em Anosov} flow on a {\em complete} Riemannian manifold of finite volume.  This time change does not change orbits and has a predictable effect on stable and unstable bundles.   One can apply all of the known machinery for Anosov flows to this rescaled flow, and transferring the information back to the original flow,  one concludes that $\varphi_t$ possesses invariant
stable and unstable foliations $\cW^u$ and $\cW^s$ that are locally uniformly $C^{1+\alpha}$.   This gives the crucial input in constructing the section $\Sigma$ and return function $R$ defined above.

In the setting of Weil-Petersson geometry, one can summarize the results of this time change:  in the exceptional case $3g - 3 + n=1$,  the Weil-Petersson geodesic flow, when run at unit speed {\em in the Teichm\"uller metric} is (like the Teichm\"uller flow) an Anosov flow.  For $3g - 3 + n>1$, the WP flow is not Anosov, even when viewed in the Teichm\"uller metric (or an equivalent Riemannian metric such as in \cite{McMo}), but it might be fruitful to study the flow from this perspective.  We remark here that Hamenst\"adt \cite{Ham} has recently constructed measurable orbit equivalences between the WP and Teichm\"uller geodesic flows 
for all $3g - 3 + n\geq 1$.

A different  approach, using anisotropic function spaces, has been employed by Liverani to establish exponential mixing for contact Anosov flows in arbitrary dimension, even when the foliations $\cW^u$ and $\cW^s$ fail to be $C^1$ \cite{Liv}.  This method is more holistic (though no less technical) as the arguments take place in the manifold itself (not a section) and  avoid symbolic dynamics. It would be interesting to attempt to import this machinery to the present  context.  This is the approach employed in the recent work of  Baladi, Demers and Liverani  on Sinai billiards in \cite{BDL} mentioned above.  

This paper is organized as follows.  In Section~\ref{s=prelim} we recall some facts about geodesic flows and basic comparison lemmas for ODEs.  In Section~\ref{s=deltaregularity}, we establish (under the hypotheses of Theorem~\ref{t=main}) $C^4$ regularity for the functions $\delta_i$ measuring distance to the cusps in $S$.  The arguments there bear much in common with standard proofs of regularity of Busemann functions in negative curvature, but additional attention to detail is required to obtain the correct order estimates on the size of the derivatives of the $\delta_i$.  In Section~\ref{s=revolution} we establish basic geometric properties of the surfaces considered here, in close analogy to properties of surfaces of revolution.  These results refine some known properties of the Weil-Petersson metric.

Section~\ref{s=global} addresses the global properties of the flow $\varphi_t$.  Here we construct a new Riemannian metric
on $T^1S$, which we call the $\star$ metric,  in which $T^1S$ is complete.  Rescaling $\varphi_t$ to be unit speed in the $\star$ metric, we obtain a new flow $\psi_t$ which we prove is Anosov, with {\em uniform} bounds on its first three derivatives (in the $\star$ metric).  We derive consequences of this, including ergodicity of $\varphi_t$ and existence and $C^{1+\alpha}$ regularity of $\varphi_t$ invariant unstable and stable foliations $\cW^s$ and $\cW^u$.

In the final section (Section~\ref{s=expmix}), we construct the section $\Sigma$ to the flow and return time function $R$ satisfying the hypotheses of the Ara\'ujo-Melbourne theorem.  In essence this is equivalent to constructing a Young tower for the return map to $\Sigma$ and is carried out using standard methods.  Here the properties of geodesics established in Section~\ref{s=revolution}  come into play in describing the dynamics of the return map of the flow to the compact part of $T^1S$.

We thank Scott Wolpert, Sebastien Gou\"ezel, Carlangelo Liverani  and Curtis McMullen for  useful conversations, and Viviane Baladi and Ian Melbourne for comments on a draft of this paper. 

\section{Notation and preliminaries}\label{s=prelim}

Let $S$ be an oriented surface endowed with a Riemannian metric.  As usual $\langle v,w\rangle$  denotes the inner product of two vectors and $\nabla$ is the Levi-Civita connection defined by the Riemannian metric. It is the unique connection that is symmetric and compatible with the metric.

 The surface $S$ carries a unique almost complex structure compatible with the metric. 
We denote this structure by $J$;  for $v \in T^1_p S$, the vector $J v$ is the unique tangent vector in $T^1_pS$ such that $(v, Jv)$ is a positively oriented orthonormal frame.

The covariant derivative along a curve $t \mapsto c(t)$ in $S$ is denoted by $D_c$, $\frac D{dt}$ or simply $'$ if it is not necessary to specify the curve; if $V(t)$ is a vector field along $c$ that extends to a vector field $\widehat V$ on $S$, we have $V'(t) = \nabla_{\dot c(t)} \widehat V$.

Given a smooth map $(s,t) \mapsto \alpha(s,t)\in S$, we let $\frac D{\partial s}$ denote covariant differentiation along a curve of the form $s \mapsto \alpha(s,t)$ for a fixed $t$. Similarly  $\frac D{\partial t}$ denotes
covariant differentiation along a curve of the form $t \mapsto \alpha(s,t)$ for a fixed $s$. The symmetry of the Levi-Civita connection means that 
 $$
 \frac D{\partial s}\frac{\partial\alpha}{\partial t}(s,t) = \frac D{\partial t}\frac{\partial\alpha}{\partial s}(s,t)
 $$
 for all $s$ and $t$.

A geodesic segment $\gamma\colon I\to S$ is a curve satisfying $\gamma''(t) = D_\gamma \dot \gamma(t) = 0$, for all $t\in I$.  Throughout this paper,  all geodesics are assumed to be unit speed: $\|\dot\gamma\|\equiv 1$.

The Riemannian curvature tensor $R$ is defined by
$$
R(A,B)C = (\nabla_A\nabla_B - \nabla_B\nabla_A - \nabla_{[A,B]})C
$$
and the Gaussian curvature $K\colon S\to \RR$ is defined by
\[K(p) =\langle R(v, Jv)Jv, v\rangle,
\]
where $v\in T^1_pS$ is an arbitrary unit vector.

For $v\in TS$, we represent each element $\xi\in T_vTS$ in the standard way as a pair
$\xi = (v_1,v_2)$ with $v_1\in T_pS$ and $v_2\in T_pS$, as follows.
Each  element $\xi\in T_vTS$ is tangent to a curve $V\colon (-1,1) \to T^1S$ with $V(0) = v$.   Let
$c = \pi\circ V\colon (-1,1) \to S$ be the curve of basepoints of $V$ in $S$, where 
$\pi\colon TS\to S$ is the standard projection.  Then $\xi$ is represented
by the pair
 $$
 (\dot c(0),  D_cV(0))  \in T_pS\times T_pS.
 $$  

Regarding $TTS$ as a bundle over $S$ in this way
gives rise to a natural Riemannian metric on $TS$, called the {\em Sasaki metric}.  In this metric, the
inner product of two elements $(v_1,w_1)$ and $(v_2,w_2)$ of $T_vTS$ is defined:
$$
\langle (v_1,w_1) , (v_2,w_2) \rangle_{Sas} = \langle v_1 , v_2 \rangle  + \langle w_1 , w_2 \rangle .
$$
This metric is induced by a symplectic form $d\omega$ on $TTS$; for vectors $(v_1,w_1)$ and 
$(v_2,w_2)$ in $T_vTS$, we have:
$$
d\omega((v_1,w_1) , (v_2,w_2)) = \langle v_1,w_2 \rangle - \langle w_1,v_2 \rangle.
$$
This symplectic form is the pull back of the canonical symplectic form on the cotangent bundle $T^*S$ by the map from $TS$ to $T^*S$ induced by identifying a vector $v \in T_pS$ with the linear function $\langle v, \cdot\rangle$ on $T_pS$.

\subsection{The geodesic flow and Jacobi fields}

For $v \in TS$ let $\gamma_v$ denote the  unique geodesic $\gamma_v$ satisfying $\dot\gamma_v(0) = v$. The geodesic flow $\varphi_t : TS\to TS$ is defined by
 $$
 \varphi_t(v) = \dot\gamma_v(t),
 $$
 wherever this is well-defined. The geodesic flow is always defined locally.

The {\em geodesic spray} is the vector field $\dot\varphi$ on $TS$ (that is, a section of $TTS$) generating the geodesic flow.  In the natural coordinates on  $TTS$ given by the connection,
we have $\dot\varphi(v) = (v,0)$, for each $v\in TS$.  The spray is tangent to the level sets $\|\cdot\| = const$.  Henceforth when we refer to the geodesic flow $\varphi_t$, we implicity mean the restriction of this flow to the unit tangent bundle $T^1S$.

  Since the geodesic flow is Hamiltonian, it  preserves a natural volume form on $T^1S$ called the Liouville  volume form.  When the integral of this form is finite, it induces a unique probability measure on $T^1S$ called the {\em Liouville measure} or {\em Liouville volume}.

Consider now a one-parameter family of geodesics, that is a map $\alpha:(-1,1)^2\to S$ with the property that
$\alpha(s,\cdot)$ is a geodesic for each $s\in (-1,1)$.  Denote by $\cJ(t)$ the vector field  
$$
\cJ(t) = \frac{\partial \alpha}{\partial s}(0,t)
$$
along the geodesic $\gamma(t) = \alpha(0,t)$.
Then $\cJ$ satisfies the {\em Jacobi equation:}
\begin{equation}\label{e=jacobi}
\cJ'' + R(\cJ,\dot\gamma)\dot\gamma = 0,
\end{equation}
in which $'$ denotes covariant differentiation along $\gamma$.
Since this  is a second order linear ODE,
the pair of 
vectors $(\cJ(0), \cJ'(0)) \in T_{\gamma(0)}M \times T_{\gamma(0)}M$ uniquely determines the vectors
$\cJ(t)$ and $\cJ'(t)$ along $\gamma(t)$. A vector field $\cJ$ along a geodesic $\gamma$ satisfying the Jacobi equation is called a {\em Jacobi field}.

The pair $(\cJ(t),\cJ'(t))$ corresponds in the manner described above to the tangent vector at $s= t$ to the 
curve $s \mapsto  \frac{\partial \alpha}{\partial t}(s,t) = \varphi_t \circ V(s)$, which is $D\varphi_t(V'(0))$. Thus
\begin{proposition} \label{prop:key}
The image of the tangent vector $(v_1,v_2)\in T_vTS$ under the derivative of the
geodesic flow $D_v\varphi_t$ is the tangent vector $(\cJ(t),\cJ'(t))\in T_{\varphi_t(v)}TS$, where $\cJ$ is the unique
Jacobi field along $\gamma$ satisfying $\cJ(0) = v_1$ and $\cJ'(0) = v_2$.
\end{proposition}

Computing the Wronskian of the Jacobi field $\dot\gamma$ and an arbitrary Jacobi field $\cJ$ shows that $\langle \cJ', \dot\gamma \rangle$ is constant.
It follows that if $\cJ'(t_0) \perp \dot\gamma(t_0)$ for some $t_0$, then $\cJ'(t) \perp \dot\gamma(t)$ for all $t$. Similarly if $\cJ(t_0) \perp \dot\gamma(t_0)$ and $\cJ'(t_0) \perp \dot\gamma(t_0)$ for some $t_0$, then $\cJ(t) \perp \dot\gamma(t)$ and $\cJ'(t) \perp \dot\gamma(t)$ for all $t$; in this  case we call $\cJ$ a {\em perpendicular Jacobi field}.
If $\alpha$ is a variation of geodesics giving rise to a perpendicular Jacobi field, then we 
call $\alpha$ a {\em perpendicular variation of geodesics}.

 The space of all perpendicular Jacobi fields along a unit speed geodesic $\gamma$ 
corresponds to the orthogonal complement $\dot\varphi^\perp(v)$ (in the Sasaki metric) to the geodesic spray $\dot\varphi(v)$ at the point $v = \dot\gamma(0) \in T^1S$.  To estimate the norm of the derivative $D\varphi_t$ on $TT^1S$, it suffices to restrict attention to vectors in the invariant subspace $\dot\varphi^\perp$; that is, it suffices to estimate the growth of perpendicular Jacobi fields along unit speed geodesics.

Because $S$ is a surface, the Jacobi equation (\ref{e=jacobi}) of a perpendicular Jacobi field along a unit speed geodesic segment can be expressed as a scalar ODE in one variable.  Given such a geodesic 
$\gamma\colon I\to S$, any perpendicular Jacobi field $\cJ$ along $\gamma$ can be written in the form
$(\cJ(t), \cJ'(t)) = (j(t)J\dot\gamma(t), j'(t)J\dot\gamma(t))$, where $j\colon I\to\RR$ satisfies the
 {\em scalar Jacobi equation:}
\begin{equation}\label{e=scalarjacobi}
j''(t) =  -K(\gamma(t)) j(t).
\end{equation}
To analyze solutions to (\ref{e=scalarjacobi}) it is often convenient to consider the functions
$u(t) = j'(t)/j(t)$ and $\zeta(t) =  j(t)/j'(t)$ which satisfy the Riccati equations
$u'(t) = -K(\gamma(t)) - u^2(t)$   and $\zeta'(t) = 1  + K(\gamma(t))\zeta^2(t)$, respectively.
In the next subsection, we describe some techniques for analyzing solutions to these types of equations.

\subsection{Comparison lemmas for Ordinary Differential Equations}\label{s=comparison}

We will use a few basic comparison lemmas for solutions to ordinary differential equations.
The first is standard and is presented without proof:

\begin{lemma}\label{l=comparison}{\bf [Basic comparison]}
Let $F\colon \RR \times  [t_0, t_1] \to \RR$ be $C^1$, and let $\zeta\colon [t_0, t_1] \to \RR$ be a solution to the ODE
\begin{equation}\label{e=basicODE}
\zeta'(t) = F(\zeta(t), t).
\end{equation}
Suppose that $\underline u,\overline u\colon [t_0,t_1]\to\RR$ are $C^1$ functions satisfying $\underline u(t_0) \leq \zeta(t_0) \leq \overline u(t_0)$.  Then the following hold:
\begin{itemize}
\item If $F(\overline u(t), t) \leq \overline u'(t)$
for all $t\in [t_0, t_1]$,  then $\zeta(t)\leq \overline u(t)$ for all $t\in [t_0,t_1]$.

\item If $F(\overline u(t), t) < \overline u'(t)$
for all $t\in [t_0, t_1]$,  then $\zeta(t) < \overline u(t)$ for all $t\in (t_0,t_1]$.

\item If $F(\underline u(t), t) \geq \underline u'(t)$
for all $t\in [t_0, t_1]$,  then $\zeta(t)\geq \underline u(t)$ for all $t\in [t_0,t_1]$.

\item If $F(\underline u(t), t) > \underline u'(t)$
for all $t\in [t_0, t_1]$,  then $\zeta(t) > \underline u(t)$ for all $t\in (t_0,t_1]$.

\end{itemize}

\end{lemma}

We will have several occasions to  deal with Riccati equations of the form
$\zeta'(t) = 1 - k^2(t)\zeta^2(t)$ on  an interval $[t_0,t_1]$ (where typically $-k^2(t) = K(\gamma(t))$ for some geodesic segment $\gamma$).
Since the curvature of the surfaces we consider is not bounded away from $-\infty$, most of the ODEs we deal with will have unbounded coefficients.  This necessitates reproving some standard results about solutions.  A key basic result is the following.

\begin{lemma}\label{l=monotone2}{\bf[Existence of Unstable Riccati Solutions]}
Suppose $k\colon (t_0, t_1] \to \RR_{>0}$ is a $C^1$ function satisfying
$\lim_{t\to t_0} k(t) = \infty$.
Then there exists a unique solution $\zeta(t)$ to the Riccati equation
\begin{equation}\label{e=Riccsimple}
\zeta' = 1 - k^2 \zeta^2
\end{equation}
for $t\in(t_0,t_1]$ satisfying $\zeta(t)>0$ on $(t_0,t_1]$ 
 and $\lim_{t\to t_0} \zeta(t)= 0$.

Moreover, if $\underline k\colon (t_0,t_1] \to \RR_{>0}$ is any $C^1$ function satisfying
$\underline k'(t) <  0$ and  $\underline k(t) \leq k(t)$, for all $t\in (t_0,t_1]$, then
$\zeta(t) \leq \underline k(t)^{-1}$,
for all $t\in [t_0,t_1]$.
\end{lemma} 

\begin{proof} Let $\underline k$ be a function satisfying the hypotheses of the lemma. 
Then there is a function $k_0: (t_0,t_1] \to \RR_{>0}$ such that $k_0'(t) < 0$ and $\underline k(t) \leq k_0(t) \leq k(t)$ for all $t \in (t_0,t_1]$ and $k_0(t)^{-1} \to 0$ as $t \to t_0$. Observe that  $(d/dt)(k_0^{-1}(t)) > 0 \geq 1 - k(t)^2k_0(t)^{-2}$ for $t_0 < t \leq t_1$.

Now fix a decreasing sequence $t_n\to t_0$ in $(t_0, t_1)$. 
For each $n > 1$ let $\zeta_n$ be the solution to (\ref{e=Riccsimple}) on $[t_n,t_1]$ with  $\zeta_n(t_n)=0$. 
We can  
apply Lemma~\ref{l=comparison} to equation (\ref{e=Riccsimple}) on the interval $[t_n,t_1]$ with  $\underline u(t) = 0$ and $\overline u(t) = k_0(t)^{-1}$. This gives us $0  < \zeta_n(t) \leq k_0(t)^{-1}$ for $t_n < t \leq t_1$.
We can also apply Lemma~\ref{l=comparison} on this interval with $\zeta = \zeta_m$ for $m \geq n$ and $\underline u = \zeta_n$. This gives us $\zeta_m(t) \geq \zeta_n(t)$, for $t_n \leq  t \leq t_1$.

The sequence of solutions $\zeta_n$ is thus increasing, positive and bounded above by $k_0^{-1}$.  It follows
that the function $\zeta := \lim_{n\to \infty} \zeta_n$ is a solution to (\ref{e=Riccsimple}), is
 positive on $(t_0, t_1]$, is bounded above by $k_0^{-1}$, and thus satisfies $\lim_{t\to t_0} \zeta(t) = 0$.

It remains to show that $\zeta$ is the only  solution of (\ref{e=Riccsimple}) with the desired properties. Suppose $\eta$ is another such solution. Since the graphs of two solutions of (\ref{e=Riccsimple}) cannot cross, we may assume that $\zeta(t) \geq \eta(t)  \geq 0$ for $t_0 \leq t \leq t_1$. But then 
 $$
(\zeta - \eta)'(t) = k(t)^2[(\eta(t)^2 - \zeta(t)^2] \leq 0
 $$
 for $t_0 < t \leq t_1$. Since $(\zeta - \eta)(t) \to 0$ as $t \to t_0$, this is possible only if $\zeta (t) = \eta(t)$ for $t_0 \leq t \leq t_1$.
\end{proof}

We call the solution of the Riccati equation  defined by the previous lemma the {\em unstable solution} on $(t_0,t_1]$.

\begin{lemma}\label{l=monotone2.5}{\bf[Comparison of Unstable Riccati Solutions]}
For $i = 1,2$, let $k_i : (t_0, t_1] \to \RR_{>0}$ be a $C^1$ function satisfying
$\lim_{t\to t_0} k_i(t) = \infty$ and let $\zeta_i: (t_0, t_1] \to \RR_{>0}$ be the unstable solution.
Suppose  
$k_1(t) \geq k_2(t)$ for all $t\in(t_0,t_1]$. Then $\zeta_1(t)\leq \zeta_2(t)$ for all $t\in[t_0,t_1]$.
\end{lemma}

\begin{proof} Suppose $\zeta_2(t'_0) \geq \zeta_1(t'_0)$ for some $t'_0 \in (t_0,t_1]$. Then we can apply Lemma~\ref{l=comparison} to the equation $\zeta' = 1 - k_1^2 \zeta^2$ with $\zeta = \zeta_1$ and $\overline u = \zeta_2$ to obtain $\zeta_2(t)\geq \zeta_1(t)$ for all $t\in[t'_0,t_1]$. 
It now suffices to show that if there is $t'_1 \in (t_0,t_1]$ such that $\zeta_1(t) \geq  \zeta_2(t)$ for all $t\in[t_0,t'_1]$, then  we must have $\zeta_1(t) =  \zeta_2(t)$ for all $t\in[t_0,t'_1]$. But if 
$\zeta_1 \geq \zeta_2 \geq 0$ on $(t_0,t_1']$ we have
 $$
 (\zeta_1 - \zeta_2)'(t) = k_2(t)^2\zeta_2(t)^2 - k_1(t)^2\zeta_1(t)^2 \leq 0
$$
 for $t_0 < t \leq t'_1$. Since $(\zeta_1 - \zeta_2)(t) \to 0$ as $t \to t_0$, this is possible only if $\zeta_1 (t) = \zeta_2(t)$ for $t_0 \leq t \leq t'_1$.
\end{proof}

\begin{lemma}\label{l=monotone3}
Let  $k\colon (0, t_1] \to \RR_{>0}$ be a $C^1$ function satisfying
$\lim_{t\to 0} k(t) = \infty$,
and let $\zeta(t)$ be the unstable solution. Let $r > 1$.

\begin{enumerate}
\item If  $k(t)^2 \geq r(r-1)/t^2$ for all $t \in (0,t_1]$, then
$ \zeta(t)  \leq  t/r$  for all $t \in (0,t_1]$. 

\item If  $k(t)^2 \leq r(r-1)/t^2$ for all $t \in (0,t_1]$, then
$ \zeta(t)  \geq  t/r$  for all $t \in (0,t_1]$.

\item  Suppose $N>0$, $0<t_2<\min\{t_1, (r-1)/N\}$ and
$$
k(t)^2 \in \left [ \frac{r(r-1)}{t^2}  - \frac Nt ,   \frac{r(r-1)}{t^2}  +  \frac Nt \right]\qquad\text{for all $t \in (0,t_2]$.}
$$
Then there exists $M>0$ such that 
$$
\zeta(t) \in \left[ \frac{t}{r} - Mt^2, \frac{t}{r} + Mt^2\right] \qquad\text{for all $t \in (0,t_2]$.}
$$
\end{enumerate}
\end{lemma} 

\begin{proof}
1~and 2. \  These follow from Lemma~\ref{l=monotone2.5} because $\zeta(t) = t/r$ is the unstable solution for
$$
\zeta' = 1 - \frac{r(r-1)}{t^2}\zeta^2.
$$

\medskip

3. \  Choose $\delta >0$ such that $0<t_2<1/2\delta<(r-1)/N$. Then,
 $$
\frac{(r - \delta t)(r - \delta  t-1)}{t^2} \leq \frac{r(r - 1)}{t^2}  - \frac Nt 
\quad\text{and}\quad
\frac{(r + \delta t)(r + \delta t -1)}{t^2} \geq \frac{r(r+1)}{t^2}  + \frac Nt
$$
for $0 < t \leq t_2$.  It  follows from parts 1 and 2 of this lemma that for each $\tau \in (0,t_2]$ we have
 $\zeta(t) \in \left[  t/{(r + \delta\tau)},   t/{(r - \delta\tau)} \right]$,
 for all $t \in (0,\tau]$. Consequently,
$ \zeta(t)  \in \left[  t/{(r +\delta t)},   t/{(r -\delta t)} \right]$,
 for all $t \in (0,t_2]$.
Now choose $M > 2\delta/r$. We then  have
$$
\frac{t}{r} - Mt^2 \leq \frac t{(r +\delta t)} \quad \text{and}\quad
\frac  t{(r - \delta t)} \leq  \frac{t}{r} + Mt^2,
$$
for $0< t \leq t_2$. We conclude that $\zeta(t) \in [t/r - Mt^2, t/r + Mt^2]$, for $0\leq t \leq t_2$.
\end{proof}

\section{Regularity of the distance $\delta$ to the cusp}\label{s=deltaregularity}

Suppose $S$ satisfies the hypotheses of Theorem~\ref{t=main}, with $k=1$.  
Before considering the global properties of the metric on $S$, we introduce local coordinates about the puncture $p_1$ and study the behavior of geodesics that remain in this cuspidal region during some time interval.

In this section and the next, we thus assume that the punctured disk $\DD^\ast$ has been endowed with an incomplete Riemannian metric, whose completion is the closed disk $\overline \DD$.  
Assume that the lift of this metric to $\widetilde{\DD^\ast}$ is {\em geodesically convex:} that is, any two points in  $\widetilde{\DD^\ast}$ can be connected by a unique geodesic in $\widetilde{\DD^\ast}$.

Let $\rho$ be the Riemannian distance metric on $\overline \DD$ and for $z\in\DD$, let
$\delta(z) = \rho(z,0)$. For $\delta_0\in (0,1)$, we denote by $\DD^\ast(\delta_0)$ the set of
$z\in\DD^\ast$ with $\delta(z) < \delta_0$.

Assume that that there exists $r>2$ such that for all $z\neq 0$ the curvature of the Riemannian metric satisfies:
\begin{equation}\label{e=Kassump}
0 > K(z) =  - \frac{r(r-1)}{\delta(z)^2} + O(\delta(z)^{-1}),
\end{equation}
and 
\begin{equation}\label{e=nablaKassump}
 \|\nabla^j  K(z)\| =    O({\delta(z)^{-2-j}}),
\end{equation}
for $j=1,2,3$.

The main result of this section establishes regularity of the function $\delta$ and estimates on the size of its derivatives.  We also introduce a function $c$ that measures the geodesic curvature of the level sets of $\delta$ and establish some of its properties.  
The results in this section establish in this incomplete,  singular setting the standard regularity properties of Busemann functions for compact, negatively curved manifolds  (see, e.g. \cite{HeintzeImHof}) -- in particular, Busemann functions for a $C^k$ metric are $C^{k-1}$.  The main techniques are thus fairly standard but require some care in the use of comparison lemmas for ODEs.  To avoid tedium, we have described many calculations in detail but have left others to the reader.

\begin{proposition}\label{p=c} The cusp distance function $\delta$ is $C^4$.  
Let $V = \nabla\delta$, and let  $c\colon \DD^\ast\to \RR_{>0}$ be the geodesic curvature function defined by
\begin{equation}\label{e=cdef}
c = \langle \nabla_{JV} V , JV\rangle. 
\end{equation}
Then:
\begin{enumerate} 
\item  $\nabla_{JV} V = [JV,V] = cJV$.
\item  for any vector field $U$:
$\nabla_UV= c \langle U,JV\rangle JV,\,\hbox{ and } \, \nabla_U JV=-c \langle U,JV\rangle V$.
\item  $c = {r}/{\delta} + O(1)$.
\item  $\| \nabla  c\| = O(\delta^{-2})$.
\item  $\| \nabla^2  c\| = O(\delta^{-3})$.
\end{enumerate}
\end{proposition}

\begin{corollary}\label{l=cuspdist} 
The  function $\delta$ satisfies:  $\|\nabla \delta\| = 1$ and $\|\nabla^i \delta\| = O(\delta^{1-i})$, for $i=2,3,4$.  
\end{corollary}

\begin{proof}  This follows from the facts: $\nabla_U \delta = \langle U, V \rangle$, 
$\nabla_U V =   c \langle U,JV\rangle JV = O(\delta^{-1})\|U\|$,
and $\|\nabla^j c\| = O(\delta^{-1-j})$, $j=1, 2$, proved in Proposition~\ref{p=c}.
\end{proof}

\begin{proof}[Proof of Proposition~\ref{p=c}]  We prove first that $\delta$ is $C^4$, in several steps.

\medskip

\noindent{\bf Step 0: $\delta$ is continuous.}  
We realize the universal cover of the punctured  disk ${\DD^\ast}$  as the strip
$\RR\times (0,1)$ with the deck transformations $(x,t)\mapsto (x+n,t)$, $n\in \ZZ$.    Endow $\RR\times (0,1)$  with the lifted metric, which is geodesically convex by assumption, and lift $\delta$ to a function $\tilde \delta$.
By assumption,  the completion of  ${\DD^\ast}$ is $\bar\DD$, and so the completion
 $\overline{\RR\times (0,1)}$ in this metric is the union of $\RR\times (0,1]$ with a single point $p^\ast$.

Since $\RR\times (0,1)$  is negatively curved and geodesically convex, it is in particular $CAT(0)$.  The $CAT(0)$ property is preserved under completion, and so  $\overline{\RR\times (0,1]}$ is also $CAT(0)$.
Thus for every 
 for every  $\tilde z_0 \in \RR\times (0,1]$, there is a unique
unit-speed geodesic from $\tilde z_0$ to $p^\ast$.   This projects to a (unique) geodesic in 
$\DD$ from $z_0$ to $0$.

Fix $z_0\in  \DD^\ast$ with lift $\tilde z_0 \in \RR\times (0,1]$, and let $\gamma_{0,z_0}\colon [0,\delta(z_0)]\to \overline \DD$ be
the unit-speed geodesic from $0$ to $z_0$ found by the previous argument.  It has the property
that  $\delta(\gamma_{0,z_0}(t))=   \rho(0, \gamma_{0,z_0}(t) ) = t$  for every $t\in [0,\delta(z_0)]$.
Let $t_n\to 0$ be a sequence of times
in $(0,\delta(z_0))$, and define a sequence of  functions
$\tilde \delta_n\colon\RR\times (0,1] \to \RR_{>0}$ by $\delta_n(\tilde z) =  \tilde\rho(\tilde z, \tilde\gamma_{0,\tilde z_0}(t_n))$. The $\tilde \delta_n$  are convex, $C^3$ away from $\tilde\gamma_{0, \tilde z_0}(t_n)$,  and $\|\nabla \delta_n\| = 1$ for all $n$. 

\begin{lemma} For every $\tilde z\in\RR\times (0,1]$ and all $m\geq n$, we have
\begin{equation}\label{e=deltadiff}|\tilde \delta_n(\tilde z) - \tilde \delta_m(\tilde z)| \leq t_n-t_m\leq t_n,
\end{equation}
\end{lemma}

\begin{proof} This follows from the triangle inequality.
\end{proof}

Since $\tilde\delta_n$ is Cauchy, it converges (locally uniformly in $\RR\times (0,1]$) to a continuous, convex function $\widehat\delta$.  Moreover $\widehat\delta(\tilde z)$ is the distance $\tilde\rho(\tilde z,p^\ast)$.  It follows that $\tilde\delta$ is continuous (and convex), and so $\delta$ is continuous.

\medskip

\noindent{\bf Step 1: $\delta$ is $C^1$.}   
Let $\tilde V_n = \nabla \tilde \delta_n$ be the corresponding sequence of radial vector fields on $\RR\times (0,1]$.

\begin{lemma}\label{l=Vn}
Fix $\tilde z\in\RR\times (0,1]$. For all $m\geq n$ sufficiently large, we have:
$$
\|\tilde V_n(\tilde z) - \tilde V_m(\tilde z)\| \leq  \frac{t_n}{\tilde \delta(\tilde z) - t_n}.
$$
Thus $\tilde V_n$ is a Cauchy sequence in the  local uniform topology.
\end{lemma}
\begin{proof}
This is a standard argument in negative curvature (in fact nonpositive curvature suffices).  
This uses that $\|\nabla \delta_n\| = 1$ for all $n$. 
\end{proof}

This lemma implies that $\delta$ is $C^1$. Let $\tilde V$ be the local uniform limit of the $\tilde V_n$: by definition, $V = \nabla \tilde \delta$.
Let $V = \nabla \delta$ be the projection of $\tilde V$ to $\DD^\ast$.
It remains to show that $V$ is $C^3$, which implies that $\delta$ is $C^4$.  

\medskip

\noindent{\bf Step 2: $V$ is $C^1$.}

\begin{lemma}\label{l=VisC1}  There exists $\delta_0>0$, such that for every $z_0\in \DD^\ast$ with $\delta(z_0) < \delta_0$,
the following holds.
For every vector field $U$,  $ \nabla_U V$ exists and is continuous in a neighborhood of $\gamma_{0,z_0}((0,\delta(z_0)])$.
Moreover: 
\begin{equation}\label{e=jvstuff} \nabla_{JV} V (\gamma_{0,z_0}(t)) = \zeta(t)^{-1} JV(\gamma_{0,z_0}(t)),
\end{equation}
for all $t\in (0,\delta(z_0)]$, where $\zeta$ is the positive solution to the Riccati equation
\begin{equation}\label{e=zetaeq}
\zeta'(t) = 1 + K(\gamma_{0,z_0}(t)) \zeta(t)^2,
\end{equation}
given by Lemma~\ref{l=monotone2}, satisfying $\zeta(0) = 0$.
\end{lemma}

\begin{proof} Fix $\delta_0>0$ (we will specify how small it must be later).  Fix $z_0$ with $\delta(z_0)\leq \delta_0$, and denote by $\gamma$ the geodesic $\gamma_{0,z_0}$.  

 For each $n$, define a perpendicular, radial variation of geodesics $\gamma_n(s,t)$ by the properties:
  $\gamma_n(0,t) = \gamma(t)$,  $\gamma_n(s,t_n) = \gamma(t_n)$   and
$$\frac{D^2 \gamma_n}{\partial s\partial t} (s,\delta_0) =  J\dot \gamma_n (s, \delta_0),$$ 
for all $s,t$ with $t\geq t_n$ (and $s$ belonging to a small, fixed neighborhood of $0$).  Let $\delta_n\colon \DD^\ast \to \RR_{>0}$ be defined by $\delta_n(z) = \rho(z, \gamma(t_n))$;  
then for $t\geq t_n$, and $s$ sufficiently small, we have
\[\delta_n(\gamma_n(s,t)) = t - t_n.\]
It follows that for any $\epsilon>0$, if $n$ sufficiently large and $t\geq  (1+2/\varepsilon)t_n$, we have
\begin{equation}\label{e=uniformtn}
\delta(\gamma_n(s,t))  \in[ (1-\epsilon)t + \epsilon t_n , t ].
\end{equation}

We have already shown (working on the universal cover) that in a neighborhood of $\gamma$, we have $\delta_n\to\delta$ and $V_n =\nabla \delta_n \to V$ uniformly on compact sets.
Let $\gamma(s,t)$ be the limiting variation of geodesics, which satisfies
$\gamma(0,t) = \gamma(t)$, and $\delta(\gamma(s,t)) = t$.  At this point we have shown that $\gamma$ is $C^1$,
with $\partial \gamma/\partial s (s,\delta_0) = JV(\gamma(s,\delta_0))$.
Note that
$V_n(\gamma_n(0,t)) =  V(\gamma_n(0,t)) = V(\gamma(t))$,
for all $n$, $t\geq t_n$.

Since $V_n\to V$, it suffices to show that 
$$\nabla_{V_n} V_n (\gamma_n(s,t) )\to  \nabla_V V(\gamma(s,t) )\, \hbox{ and }  \nabla_{JV_n} V_n (\gamma_n(s,t)) \to  \nabla_{JV} V (\gamma(s,t)).
$$

The proof that $\nabla_{V_n} V_n (\gamma_n(s,t)) \to  \nabla_V V(\gamma(s,t))$ is immediate:
since $\gamma_n$ is a variation of geodesics, we have
$\gamma_n'' = \nabla_{V_n}V_n \equiv 0 \equiv \nabla_{V}V = \gamma''$.

We now show that $\nabla_{JV_n} V_n (\gamma_n(s,t)) \to  \nabla_{JV} V (\gamma(s,t))$.
 Let $j_n(s,t)$ be the scalar Jacobi field associated with the perpendicular variation $\gamma_n$:
$$
\frac{\partial\gamma_n}{\partial s} = j_n(s,t)J \dot\gamma_n(s,t) = j_n(s,t)J V_n(s,t).
$$
On the one hand,
$$ \frac{D}{\partial s} V_n(\gamma_n(s,t)) =  \frac{D^2}{\partial s\partial t} \gamma_n(s,t) = j_n'(s,t) J V_n(s,t),$$
while on the other hand,
$$ \frac{D}{\partial s} V_n(\gamma_n(s,t)) = \nabla_{j_n(s,t)J V_n(s,t)} V_n = j_n(s,t) \nabla_{J V_n} V_n (\gamma_n(s,t)).
$$

Writing $\zeta_n(s,t) = j_n(s,t)/j_n'(s,t)$, we thus have
$ \zeta_n(s,t)  \nabla_{J V_n} V_n (\gamma_n(s,t)) = J V_n (\gamma_n(s,t))$.
We prove that $\zeta_n(s,t)$ converges to $\zeta(s,t)$, the positive solution to (\ref{e=zetaeq}).

To see this,  we first establish uniform upper and lower bounds for $\zeta_n(s,t)$, for $t\geq (1+2/\epsilon)t_n$.  The  $\zeta_n$ satisfy the Riccati equation:
\begin{equation}\label{e=zetan}
\zeta_n'(s,t) = 1 + K(\gamma_n(s,t)) \zeta_n(s,t)^2,
\end{equation}
with $\zeta_n(s,t_n) = 0$, for all $s$.
Now $$K(\gamma_n(s,t)) = -\frac{r(r-1)}{\delta(\gamma_n(s,t))^2} + O(\delta(\gamma_n(s,t))^{-1}).$$
By (\ref{e=uniformtn}) we  thus have:
\begin{equation}\label{e=crudeKbounds}
K(\gamma_n(s,t)) \in \left[-\frac{r^2}{t^2} , -\frac{(r-1)^2}{ t^2} \right],
\end{equation}
if $\delta(z_0) \leq \delta_0$ is sufficiently small,  $t\geq (1+2/\epsilon) t_n$, and $n$ is sufficiently large. 

We show that there exists $\mu\in (0,1)$ such that,
for $n$ sufficiently large, we have 
\begin{equation}\label{e=zetamu}
\zeta_n (t)\in [\mu(t-t_n), \mu^{-1}(t-t_n)].
\end{equation}
To see the lower bound, let $u = \mu(t-t_n)$.  Then $u' = \mu$.  On the other hand,
when $\zeta_n(t) = u$, we have $\zeta_n' = 1 + K\mu^2(t-t_n)^2$. This is larger than $u'$ provided that
$1 + K\mu^2(t-t_n)^2 > \mu$.  But this will hold if
$1  - r^2\mu^2 > \mu$.
The upper bound is similarly obtained.
By Lemma~\ref{l=comparison}, $\zeta_n(t) \geq \mu(t-t_n)$ for all $t\geq t_n$, which proves (\ref{e=zetamu}).

We now use  Lemma~\ref{l=comparison} to prove that for some large but fixed $C$:
\begin{equation}\label{e=zetamzetan}
|(\zeta_{m} - \zeta_n)(s,t)|  \leq  C t_n,
\end{equation}
for all $s$ and $t\geq (1+2/\epsilon) t_n$.
For $m\geq n$, we have, since $\zeta_n(s, t_n) = 0$:
\[|(\zeta_{m} - \zeta_n)(s,t_n)| \leq |\zeta_{m}(s,t_n)| + |\zeta_n(s,t_n)| = |\zeta_{m}(s,t_n)| \leq t_n.\]
Subtracting the ODEs for $\zeta_m$ and $\zeta_n$, we have for $t>t_n$:
\[
(\zeta_n - \zeta_m)'(s,t) = K(\gamma_n(s,t)) \zeta_n(s,t)^2 - K(\gamma_m(s,t)) \zeta_m(s,t)^2 = 
\]
\[ \left( K(\gamma_n(s,t))   - K(\gamma_m(s,t)) \right) \zeta_m(s,t)^2 + 
  K(\gamma_n(s,t)) \left(\zeta_n(s,t)^2 -   \zeta_m(s,t)^2 \right )
\]
\[
= O\left( \frac{t_n (t-t_n)^2}{t^3}\right)  +   K(\gamma_n(s,t))(\zeta_n(s,t)+\zeta_m(s,t))(\zeta_n(s,t) - \zeta_m(s,t)),
\]
since $\|\nabla K(\gamma_n(s,t)) \|= O(\delta(\gamma_n(s,t))^{-3}) = O(t^{-3}$),
$\rho(\gamma_n(s,t), \gamma_m(s,t)) = O(t_n)$, and $\zeta_m(s,t) = O(t-t_m)$.
Writing $y = \zeta_n - \zeta_m$, we have that $y$ satisfies the ODE 
\begin{equation}\label{e=y}
y' = O\left( \frac{t_n(t-t_n)^2}{t^3}\right)  +   K(\gamma_n(s,t))(\zeta_n(s,t)+\zeta_m(s,t))y.
\end{equation}
Fix $n$ and let $u(t) = C t_n$.  Then $u'(t) = 0$, and $y'$ evaluated at $y=u$ is
\[
y' = O\left( \frac{t_n (t-t_n)^2}{t^3}\right)  +   K(\gamma_n(s,t))(\zeta_n(s,t)+\zeta_m(s,t))C t_n
\]
We claim that if $C$ is sufficiently large, then $y' \leq 0 = u'$.   To see this fix $N>0$ such that
\[
y' \leq  \frac{N t_n (t-t_n)^2}{t^3} +   K(\gamma_n(s,t))(\zeta_n(s,t)+\zeta_m(s,t))y.
\]
Then (\ref{e=crudeKbounds}) and (\ref{e=zetamu}) imply that
\[
y' \leq  \frac{N t_n (t-t_n)^2}{t^3}  -   2 \frac{(r-1)^2}{t^2} \mu(t-t_n) C t_n
\leq \frac{ t_n (t - t_n)}{t^2} \left( N  - 2 (r-1)^2 \mu C  \right).
\]
Clearly this is $\leq 0$ for $t\geq t_n$  if $C$ is sufficiently large.
Since $y(t_n) \leq  t_n\leq u$, for all $C\geq1$,
this implies by Lemma~\ref{l=comparison} that $y \leq u$ for all $t\geq (1+2/\epsilon)t_n$; a similar argument shows that $y\geq -u$, and hence if $C$ is sufficiently large and $m\geq n$, then for all $t\geq (1+2/\epsilon)t_n$ (\ref{e=zetamzetan}) holds.

Thus $|(\zeta_{m} - \zeta_n)(s,t)|$ tends to $0$ as $t_n\to 0$, with $s,t$ fixed.  Thus the $\zeta_n(s,t)$ converge, and since they satisfy 
(\ref{e=zetan}),
their limit $\zeta(s,t)$ satisfies (\ref{e=zetaeq}).
We obtain that the functions $\zeta_n(s,t) \nabla_{JV_n} V_n (\gamma_n(s,t)) =  JV_n(\gamma_n(s,t))$
converge locally uniformly, and hence $\nabla_{JV} V$ exists and is continuous.

 Since
$\nabla_{JV_n} V_n (\gamma_n(s,t)) = \zeta_n(s,t)^{-1} JV_n(\gamma_n(s,t))$,
 $\zeta_n\to \zeta$, and $JV_n\to JV$, 
we obtain (\ref{e=jvstuff}) by taking a limit and setting $s=0$.\end{proof}

In light of Lemma~\ref{l=VisC1}, we define a function $\nu\colon \DD^\ast\to \RR_{>0}$
as follows.  For each $z\in \DD^\ast$, let $\zeta\colon [0,\delta(z)]\to \RR_{>0}$ be the positive solution to
(\ref{e=zetaeq}) given by Lemma~\ref{l=monotone2}.  We then set $\nu(z):= \zeta(\delta(z))$.
It follows immediately from Lemma~\ref{l=VisC1} that for every $z\in \DD^\ast$,  we have
\begin{equation}\label{e=nuproperty}
\nabla_{JV} V(z)  = \nu(z)^{-1} JV(z).
\end{equation}

\medskip

\noindent{\bf Step 3: $\nu$ is $C^1$.}

To prove that $\delta$ is $C^3$, it thus suffices to show that $\nu$ is $C^1$.  
Equation (\ref{e=zetaeq})
implies that $\nabla_V\nu$ exists and is continuous.
It remains to show that 
 $\nabla_{JV}\nu$ exists and is continuous.

We fix $\delta_0>0$ as above and let $z_0\in \DD^\ast(\delta_0)$,  and  we reintroduce the variations of geodesics
 $\gamma_n(s,t)$ defined by the properties:
  $\gamma_n(0,t) = \gamma_{0,z_0}(t)$,  $\gamma_n(s,t_n) = \gamma_{0,z_0}(t_n)$   and 
$$\frac{D^2 \gamma_n}{\partial s\partial t} (s,\delta_0) =
j_n'(s,\delta_0)J\dot \gamma_n (s, \delta_0),$$
for all $s$.  As above, write $\gamma = \gamma_{0,z_0}$, and let $\gamma(s,t)$
be the limiting variation of geodesics.  
We observe that Lemma~\ref{l=VisC1} also implies that $j_n\to j$ and $j_n' \to j'$, locally uniformly,
where $\partial \gamma/\partial s = j JV$ and $D^2 \gamma/\partial s\partial t = j' JV$.
The convergence follows from the formulae:
\[
j_n'(s,t) =  \exp\left( \int_{t}^{\delta_0} K(\gamma_n(s,x))\zeta_n(x)\, dx\right), \;\, j'(s,t) =  \exp\left( \int_{t}^{\delta_0} K(\gamma(s,x))\zeta(x)\, dx\right),
\]
\[
j_n(s,t) = \zeta_n(s,t) j_n'(s,t), \quad\hbox{ and }\, j(s,t) = \zeta(s,t) j'(s,t).
\]
Thus the variation $\gamma$ is  $C^2$ on $\DD^\ast$, and  satisfies $\gamma(s,0) = 0$.
We record here a lemma, which follows easily from these formulae, combined with (\ref{e=deltadiff}), (\ref{e=crudeKbounds}) and (\ref{e=zetamzetan}).
\begin{lemma}\label{l=jratio} For all $t\geq t_n$, we have
${j_n'(t)}/{j_m'(t)} = 1+ O\left({t_n}/{t}\right)$,
and for all  $t\geq (1+2/\epsilon) t_n$, we have
${j_n(t)}/{j_m(t)} = 1+ O\left({t_n}/{t}\right)$.
\end{lemma}

As above, let $\zeta_n(s,t)$ be the solution to (\ref{e=zetan}), and let $\zeta(s,t)$ be the solution to  (\ref{e=zetaeq}).
Note that $\nu(\gamma(s,t)) = \zeta(s,t)$.

 To prove that $\nabla_{JV}\nu$ exists and is continuous, we show that $\partial_s\zeta_n(s,t)$ converges uniformly to $\partial_s\zeta(s,t)$, the unique bounded solution to
\begin{equation}\label{e=partialszeta}
\partial_s\zeta'(s,t) = \partial_s K(\gamma(s,t)) \zeta(s,t)^2 + 2\zeta(s,t) K (\gamma(s,t)) \partial_s\zeta(s,t),
\end{equation}
which satisfies $\partial_s\zeta(s,0) = 0$, for all $s$.
Since $\partial_s\zeta = j(s,t) \nabla_{JV} \zeta$
and $j_n\to j >0$ locally uniformly, this will imply that $\nabla_{JV} \zeta$ exists and is continuous.

\begin{lemma}\label{l=partialzetadiff}  There exists $M>0$ such for all  $ m\geq n$ and all $t\geq (1+2/\epsilon) t_n$, we have
\begin{equation}\label{e=partialzetadiff}
|\partial_s\zeta_n (s,t) -  \partial_s\zeta_m(s,t)| \leq  M t_n j_n'(s,t).
\end{equation}
\end{lemma}

\begin{proof}  
Differentiating equation (\ref{e=zetan}) with respect to $s$, we obtain:
\begin{equation}\label{e=partialzetan}
\partial_s\zeta_n'(s,t) = \partial_s K(\gamma_n(s,t)) \zeta_n(s,t)^2 + 2\zeta_n(s,t)  K(\gamma_n(s,t))\partial_s\zeta_n(s,t).
\end{equation}
Note that since $\zeta_n(s,t_n) = 0$, for all $n$, we have that $\partial_s\zeta_n(s,t_n) = 0$, for all $n$.

To simplify notation, fix $s$, and write $w_n(t) = \partial_s\zeta_n(s,t)$,
 $u_n(t) =  \zeta_n(s,t)$, and $u=\zeta(s,t)$.   Then $u_n(t_n) = 0$, for all $n$.  From equations (\ref{e=zetamu}) and  (\ref{e=zetamzetan}),  we have $|u_n| \leq C t$ and    $|u_n - u_m| \leq Ct_n$. Then equation (\ref{e=partialzetan}) gives
\begin{equation}\label{e=partialzetanshort} 
w_n' =  j_n  (\nabla_{JV_n} K ) u_n^2 + 2 u_n  (K\circ \gamma_n )w_n
\end{equation}

We first claim there exists $C>0$ such that $|w_n(t)|\leq C j_n(t)$, for all $t\geq t_n$.
Let $y = C j_n(t)$. Then $y' = Cj_n'(t)$,
whereas $w_n'$ evaluated at  $w_n = y$ gives $w_n' =  j_n (\nabla_{JV_n} K ) u_n^2 +  2 C u_n  (K\circ \gamma_n ) j_n$.
Then $w_n' \leq y'$ if and only if
$ j_n (\nabla_{JV_n} K)  u_n^2 +  2 C u_n  (K\circ \gamma_n ) j_n \leq Cj_n'$;
dividing through by $j_n'$, and recalling that $u_n = j_n/j_n'$, we are reduced to showing:
\[   \nabla_{JV_n} K  u_n^3 +  2 C u_n^2  (K\circ \gamma_n )  \leq C,
\]
which holds if and only if
$C(1- 2 u_n^2  (K\circ \gamma_n ) ) \geq  ( \nabla_{JV_n} K)  u_n^3$.
Since $u_n = O(t)$,  $K<0$, and  $\| \nabla_{JV_n} K \|=  O(\delta^{-3}) =O(t^{-3})$, for $n$ sufficiently large, this will hold provided that $C$
and $n$ are sufficiently large.  We conclude that
\begin{equation}\label{l=wnbound}
|w_n(t)| \leq C  j_n (t),
\end{equation}
for all $t\geq t_n$.

We next claim that  there exists $M>0$ such that for $m\geq n$ sufficiently large and $t\geq (1+2/\epsilon) t_n$, we have
\begin{equation}\label{e=wndiff}
|w_n(t) - w_m(t)| \leq  \ M t_n j_n'(t)
\end{equation}
This will give the conclusion (\ref{e=partialzetadiff}) of Lemma~\ref{l=partialzetadiff}.

For $m\geq n$, subtracting the corresponding equations in (\ref{e=partialzetanshort}), we obtain:
\begin{equation}\label{e=wnminuswm}
(w_n - w_m)'=j_n ( \nabla_{JV_n} K)  u_n^2  -   j_m ( \nabla_{JV_m} K)  u_m^2  + 2 u_n  (K\circ \gamma_n )w_n -  2 u_m  (K\circ \gamma_m )w_m.
\end{equation}

We claim that  there exists $N>0$  such that for  $m\geq n$, and 
$t\geq  (1+2/\epsilon) t_n$, we have
\begin{equation}\label{e=epsilonclaim}
|(w_n - w_m)'  -  2 u_n  (K\circ \gamma_n ) (w_n - w_m)|   \leq \frac{N t_n j_n }{t^2}.
\end{equation}

Assuming this claim, let us complete the proof of  (\ref{e=wndiff}). Let $N$ be given so that  (\ref{e=epsilonclaim}) holds for $m\geq n$, and 
$t\geq (1+2/\epsilon)t_n$.  Let
$y(t) = w_n(t) - w_m(t)$. Then $|y' - 2u_n  (K\circ \gamma_n ) y| \leq \epsilon j_n/t^2$, and
  (\ref{l=wnbound}) implies that
\[y((1+2/\epsilon)t_n) = |w_n( (1+2/\epsilon)t_n) - w_m( (1+2/\epsilon) t_n)| \leq |w_m((1+2/\epsilon)t_n)| + |w_n((1+2/\epsilon)t_n)| 
\]
\[  \leq  C( j_m ((1+2/\epsilon) t_n)  +  j_n ((1+2/\epsilon) t_n))\] \[ \leq  C\left( j_m'((1+2/\epsilon)t_n) u_m( (1+2/\epsilon)t_n)  +j_n'((1+2/\epsilon)t_n) u_n( (1+2/\epsilon)t_n) \right)\] \[< N t_n j_n'((1+2/\epsilon)t_n),\]
for some $N>0$, since
$u_m( (1+2/\epsilon)t_n) , u_n((1+2/\epsilon) t_n) \leq  2\epsilon^{-1}\mu^{-1} t_n$, by (\ref{e=zetamu}), and
$j_n'(t)/j_m'(t) =  1 + O(t_n/t)$, by Lemma~\ref{l=jratio}. This shows that  (\ref{e=wndiff}) holds at $t= (1+2/\epsilon)t_n$, provided $n$ is sufficiently large.

 We claim that there exists $M>0$ such that
for all such $m,n$, we have $|y(t)| \leq M t_n  j_n'(t) $, for $t\geq (1+2/\epsilon)t_n$.
We prove the upper bound; the lower bound is similar. We will employ Lemma~\ref{l=comparison}.

To this end, let  $z(t) = M t_n j_n'(t) $.   Note that
$z' = Mt_n j_n'' = -M t_n (K\circ\gamma_n) j_n$,
whereas  evaluating $y'$ at $y=z$, we get
\[
y'(t) \leq  2 M t_n u_n  (K\circ \gamma_n ) j_n'  + \frac{Nt_n j_m}{t^2} =  2M t_n  (K\circ \gamma_n ) j_n + \frac{N t_n j_m}{t^2}.
\]
To satisfy the hypotheses of Lemma~\ref{l=comparison}, we require that $y'(t)\leq z'(t)$ whenever $y=z$, which is implied by:
\[
2M t_n  (K\circ \gamma_n ) j_n + \frac{N t_n j_n}{t^2} \leq -M t_n (K\circ\gamma_n) j_n,
\]
or:
\[
\frac{N t_n j_n}{t^2} \leq -3 M t_n (K\circ\gamma_n) j_n.
\]
Since $-K\circ\gamma_n(t)  \geq (r-1)^2/t^2$ (by \ref{e=crudeKbounds}), we see that this will hold (for all $n$ sufficiently large) if $M> N/ 3(r-1)^2$.
This establishes (\ref{e=wndiff}).

We return to the proof of the claim that there exists an $N>0$ such that for $m\geq n$, and 
$t\geq (1+2/\epsilon) t_n$ the inequality  (\ref{e=epsilonclaim}) holds.  The proof amounts to adding and subtracting terms 
within the left hand side of (\ref{e=epsilonclaim}), varying one at a time the multiplied quantity in each term. The added and subtracted terms are grouped in twos and the absolute value of the difference in each pair is bounded above.   To illustrate, consider the difference appearing on the left hand side of (\ref{e=epsilonclaim}).  The first two terms appearing in that difference, coming from (\ref{e=wnminuswm}), are:
\[
j_n ( \nabla_{JV_n} K)  u_n^2  -   j_m ( \nabla_{JV_m} K)  u_m^2 \qquad \qquad \qquad \qquad \qquad \qquad \qquad \qquad\qquad \qquad \qquad \]
\[ \qquad\qquad \qquad=\left( j_n ( \nabla_{JV_n} K)  u_n^2  - j_m ( \nabla_{JV_n} K)  u_n^2 \right) + \left( j_m ( \nabla_{JV_n} K)  u_n^2 -   j_m ( \nabla_{JV_m} K)  u_m^2 \right).
\]
The quantity $ j_n ( \nabla_{JV_n} K)  u_n^2  - j_m ( \nabla_{JV_n} K)  u_n^2$ can be bounded,  and the remaining term $j_m ( \nabla_{JV_n} K)  u_n^2 -   j_m ( \nabla_{JV_m} K)  u_m^2 $ can be further decomposed, as follows.  First, using (\ref{e=zetamu}) to bound $|u_n|$, the assumption that $\|\nabla K\| = O(\delta^{-3})$ together with (\ref{e=uniformtn}) to bound $\|\nabla_{JV_n} K\|$, and the fact from Lemma~\ref{l=jratio} that $(j_m - j_n) = j_n O(t_n/t)$, we have that
\[
|j_n ( \nabla_{JV_n} K)  u_n^2  - j_m ( \nabla_{JV_n} K)  u_n^2| = 
j_n  u_n^2 \left| \nabla_{JV_n} K  \right|  \left| 1  - \frac{j_m}{j_n}\right| \leq j_n O\left(\frac{t_n }{t^2}\right) .
\]
Second, to deal with the remaining term $j_m ( \nabla_{JV_n} K)  u_n^2 -   j_m ( \nabla_{JV_m} K)  u_m^2$, we write:
\[
j_m ( \nabla_{JV_n} K)  u_n^2 -   j_m ( \nabla_{JV_m} K)  u_m^2 \qquad \qquad \qquad \qquad \qquad \qquad \qquad \qquad\qquad \qquad\qquad  \]
\[ \qquad \qquad= \left( j_m ( \nabla_{JV_n} K)  u_n^2 -   j_m ( \nabla_{JV_n} K)  u_m^2 \right) + \left(    j_m ( \nabla_{JV_n} K)  u_m^2 -  j_m ( \nabla_{JV_m} K)  u_m^2  \right),
\]
and bound each term separately in a similar way to give a bound on the initial quantity $|j_n ( \nabla_{JV_n} K)  u_n^2  -   j_m ( \nabla_{JV_m} K)  u_m^2| $ of order $j_n t_n/t^2$.   The same procedure is used to bound the remaining part of the difference
appearing in (\ref{e=epsilonclaim}), which is:
\[
|\left(2 u_n  (K\circ \gamma_n )w_n -  2 u_m  (K\circ \gamma_m )w_m\right)  -  2 u_n  (K\circ \gamma_n ) (w_n - w_m) | \qquad \qquad \qquad \qquad
\]
\[
 \qquad \qquad \qquad \qquad \qquad \qquad= |-  2 u_m  (K\circ \gamma_m )w_m + 2 u_n  (K\circ \gamma_n )w_m|.
\]
 In all, we use the following  bounds:
\begin{itemize}
\item $\delta$ is comparable to $t$,  by (\ref{e=uniformtn});
\item $u_n = O(t)$, by  (\ref{e=zetamu});
\item $|u_n - u_m|  = O( t_n)$,  by (\ref{e=zetamzetan});
\item $|j_n - j_m| =   O(j_n t_n/t)$, for $t\geq (1+2/\epsilon)t_n$, by Lemma~\ref{l=jratio};
\item $|K\circ\gamma_n - K\circ \gamma_m| = O(t_n t^{-2})$, since $|K| = O(t^{-2})$ and $d(\gamma_n,\gamma_m) = O(t_n)$;
\item $|\nabla_{JV_n} K - \nabla_{JV_m} K|  = O(t_n t^{-3})$, which uses the bounds on $\|\nabla K\|$ and $\|\nabla^2 K\|$, as well as  (\ref{e=uniformtn}) and Lemma~\ref{l=Vn}.
\end{itemize}
The details are left to the patient reader. 

\bigskip

 This finishes the verification of the claim in (\ref{e=epsilonclaim}), and thus the proof of Lemma~\ref{l=partialzetadiff}.\end{proof}

To finish the proof that $\nu$ is $C^1$, note that equation (\ref{e=partialzetadiff}) can then be re-expressed using (\ref{e=zetamu}):
\begin{equation}\label{e=wndiff2}
|\partial_s\zeta_n(s,t) - \partial_s\zeta_m(s,t)| \leq M t_n \zeta_n^{-1}(s,t) j_n(s,t) \leq \frac{M t_n j_n(s,t)}{\mu(t-t_n)} \leq \frac{2 M t_n j_n(s,t)}{\mu t}, 
\end{equation}
for $t\geq (1+2/\epsilon) t_n$.
Recalling that $j_n\to j$, 
we conclude that $\partial_s \zeta_n(s,t) \to \partial_s \zeta(s,t)$ locally uniformly in $s,t$.  The bounds
$|\partial_s \zeta_n(s,t)| \leq C j_n(s, t)$ from (\ref{l=wnbound}) become in the limit 
$|\partial_s \zeta(s,t)| \leq C j(s,t)$.  But $\partial_s\zeta(s,t) = j(s,t) \nabla_{JV} \zeta$, and so we conclude that
\begin{equation}\label{e=JVnubound}
|\nabla_{JV}\nu| \leq C.
\end{equation}

 \noindent{\bf Step 4: $\nu$ is $C^2$.}
A very similar proof to the one in Step 3 (with more terms to estimate, but using, in addition to the previously obtained bounds, the bound $\|\nabla^3 K\| = O(\delta^{-5})$) gives that $\nu$ is $C^2$ with $\|\nabla^2\nu\| = O(\delta^{-1})$. One obtains this estimate as in the previous step by bounding $\|\nabla_{V}^2\nu\|$, $\|\nabla_{V}\nabla_{JV}\nu\|$, $\|\nabla_{JV}\nabla_V\nu\|$, and $\|\nabla_{JV}^2\nu\|$.  Each of these is controlled by a  differential equation, whose solutions can be estimated using a double variation of geodesics
$\gamma(s_1,s_2,t)$.  The key point, illustrated by the previous computations,  is that because $K$ has the ``expected" order of derivatives with respect to $\delta$, any quantity obtained by solving a first-order linear differential equation derived from the Riccati equation with coefficients expressed in terms of these derivatives will have the ``expected" order in $\delta$ as well.  Thus,
$\|\nabla^i K\| = O(\delta^{-2-i})$ for $i=1,2,3$ implies that $\|\nabla^{i} \nu\| =  O(\delta^{1-i})$, for $i=1,2$.

This completes the proof that $\delta$ is $C^4$. 
We now turn to items 1-5.

\medskip

1. \ $\nabla_{JV}V - \nabla_V JV = [JV,V]$ by the symmetry of the Levi-Civita connection. But $\nabla_{V} JV =0$.

2. \  For arbitrary $U$, we have  $\nabla_U V=\langle U,V\rangle  \nabla_VV+\langle U,JV\rangle \nabla_{JV} V= c \langle U,JV\rangle JV$, giving the first conclusion: $\nabla_UV= c \langle U,JV\rangle JV$. The second conclusion follows from the first and the fact that $J$ is parallel.

3. \   Note that $\nu = c^{-1}$.  It then is equivalent to prove that
$\nu = {\delta}/{r} + O(\delta^2)$.
Fix $z$ and let $\gamma =  \gamma_{0,z}$.  Along this geodesic, $\zeta(t) = \nu(\gamma(t))$ satisfies the equation (\ref{e=zetaeq}).
On the other hand $-K(\gamma(t)) = r(r-1)/t^2 + O(t^{-1})$.  Part 3 of Lemma~\ref{l=monotone3} implies the desired result.
  
4. \   
The desired estimate $\nabla_{JV}c = O(\delta^{-2})$ is equivalent to $\nabla_{JV}\nu = O(1)$ because $\nabla_{JV}\nu = -c^{-2}\nabla_{JV}c$ and $c= O(\delta^{-1})$.   But (\ref{e=JVnubound}) gives that  $\nabla_{JV}\nu = O(1)$.

5. \   
The estimate $\|\nabla^2 c\| = O(\delta^{-3})$ is equivalent to $\|\nabla^2\nu\| = O(\delta^{-1})$.  This estimate was
obtained in Step 4 above. \end{proof}

\section{Geometry of the cusp: commonalities with surfaces of revolution}\label{s=revolution}

We continue to work locally in $\DD^\ast$ with a metric satisfying (\ref{e=Kassump}) and
(\ref{e=nablaKassump}).  In this section we establish properties of geodesics in this cuspidal region.  The theme of this section is that metrics of this form inherit many of the geodesic properties of 
a surface of revolution for a profile function $y= x^r$, with $r>2$.
In $\RR^3$, coordinates on this surface are 
\[(x, \phi)\mapsto (x, x^r\cos\phi, x^r\sin\phi), \quad x\in (0,1], \,\phi\in [0,2\pi].\]
 As remarked in the introduction, if $\delta$ denotes the distance to the cusp $(0,0,0)$ on this surface, then
$\delta(x,\phi)  = x + o(x)$ and (\ref{e=Kassump})  holds.  Other properties are:
\begin{itemize}
\item {\bf Area:}  The area of the region $\{\delta \leq t_0\}$ is $2 \pi (r+1)^{-1} t_0^{r+1} + O( t_0^{r+2})$.
\item {\bf Clairaut Integral:} Let $\gamma(t)$ be a geodesic segment in the surface of revolution, and let  $\theta(t)$ be the angle between $\dot\gamma(t)$ and the foliation $\{\phi = const.\}$.  Then
the function $t\mapsto  x(\gamma(t))^r \sin\theta(t)$ is constant.
\end{itemize}
We establish here in Sections~\ref{s=area}-\ref{s=QC} the analogues of these properties in our setting.  We also establish in Section~\ref{s=jacobi} some coarse invariance properties of positive Jacobi fields in $\DD^\ast$.

\subsection{Area}\label{s=area}

Fix $\delta_0>0$, and for $t_0 \leq \delta_0$,  denote by $\DD^\ast(t_0)$ the disk $\delta\leq t_0$.  Let $c$ be the geodesic curvature function defined in the previous section, $d\ell$ the arclength element and $d\vol$ is the area element defined by the metric.  

Let  $\gamma(s,t)$ be the radial unstable variation of geodesics
described in the previous section, defined by the properties
\begin{itemize}
\item $\gamma(s,0) = 0$, for all $s$,
\item $\gamma'(s,t) = V(\gamma(s,t))$, for all $s,t$, and
\item $\partial_s\gamma(s,\delta_0) = JV(\gamma(s,\delta_0))$, for all $s$.
\end{itemize}
The arclength element is found by differentiating  $\gamma(s,t)$
with respect to $s$:
\[d\ell(\gamma(s, t)) = {\partial_s \gamma} \, ds = {j(s,t)}\,d s.\]
Using part 3 of Proposition~\ref{p=c}, we estimate $j(s,t)$ by
 \[j(s,t) = j(s,\delta_0) \exp\left(\int_{t}^{\delta_0}  -c(\gamma(s,x))\, dx \,\right) = \exp\left(\int_{\delta_0}^t  (\frac{r}{x} + O(1))\, dx \,\right) 
=  \frac{t^r}{\delta_0^r} + O(t^{r+1}),\]
and so $d\ell(\gamma(s,t)) =  ({t^r}\delta_0^{-r} + O(t^{r+1}))  \, ds$.
We obtain that:
\begin{equation}\label{e=volume}
d\vol(\gamma(s, t)) = \left(  \frac{t^r}{\delta_0^{r}} + O(t^{r+1}) \right) ds\,dt.
\end{equation}
The volume of the region $\DD^\ast(t_0)$ is obtained by integrating 
$d\vol(\gamma(s, t))$ over the region $s\in[0,\ell(\delta_0)]; t\in [0,t_0]$,
where  $\ell(\delta_0)$ is the length
of the curve $\delta = \delta_0$.
It follows that
\[\vol(\DD^\ast(t_0)) =   \frac{\ell(\delta_0) t_0^{r+1}}{(r+1)\delta_0^{r}} + O(t_0^{r+2}).
\]

\subsection{The angular cuspidal  functions $a$ and $b$}

 For $v\in T^1 \DD^*$,  we define $a(v), b(v)$ by:
\begin{equation}\label{e=abdef}
a(v) = \langle v, V\rangle,\,\hbox{ and  }\, b(v) = \langle v, JV \rangle.
\end{equation}
Thus the vectors $v$ with $b(v)=0$ and $a(v) = -1$ point directly at the cusp $0$ -- that is, the geodesics that they determine hit the cusp in finite time -- and the vectors with
 $b(v)=0$ and $a(v) = 1$ point away from  $0$.

The functions $a,b\colon T^1\DD^*\to [-1,1]$ satisfy $a^2 + b^2 = 1$; in the example of the surface of revolution mentioned in the beginning of the section we have  $a(v)=\cos\theta(v)$ and $b=\sin\theta(v)$, where $\theta(v)$ is the angle 
between $v$ and the foliation $\{\phi = const.\}$, measured from the direction pointing into the cusp.
Recall the definition  $c(p) = \langle \nabla_{JV} V(p), JV(p)\rangle_p$ from (\ref{e=cdef}).  We study here how the functions
$a, b$ and $c$ vary along a geodesic in $\DD^\ast$.

\begin{lemma} \label{l=derivs}  Let $\gamma\colon[0,T]\to \DD^\ast$ be a geodesic segment, and for $t\in[0,T]$,
write $a(t) = a(\dot\gamma(t)), b(t) = b(\dot\gamma(t))$, and $c(t) = c(\gamma(t))$.  Then:
\begin{enumerate}
\item $\delta' = a$;
\item $a' = b^2c = r b^2/\delta + O(b^2)$;
\item $b' = -abc  = -rab/\delta + O(ab)$;
\end{enumerate}

\end{lemma}

\begin{proof} This is a straightforward verification. From the definitions, we have $\delta' = \langle \dot \gamma,  \nabla \delta\rangle = \langle \dot \gamma,  V \rangle = a$, and
$a'=\nabla_{\dot\gamma}\langle \dot\gamma,V\rangle=\langle\dot\gamma,\nabla_{\dot\gamma}V\rangle=\langle\dot\gamma, b c  JV\rangle= b^2 c$.
Similarly,
$b' =\langle\dot\gamma,\nabla_{\dot\gamma}JV\rangle=-\langle\dot\gamma,b c V\rangle=-abc$. We apply conclusion 3 of Proposition~\ref{p=c} to get the final estimate.
\end{proof}

\subsection{Quasi-Clairaut Relations}\label{s=QC}

We next prove that there is a Clairaut-type integral for geodesic rays in $\DD^*$.
Recall that  for $\delta_0\in(0,1)$,  $\DD^\ast(\delta_0)$ denotes the set of
$z\in\DD^\ast$ with $\delta(z) \leq \delta_0$. 

\begin{proposition}
\label{p=Clairaut}
If $\delta_0$ is sufficiently small,  then  there exists $C=1+O(\delta_0)$, such that for every geodesic segment
$\gamma_v\colon [0,T]\to \DD^\ast(\delta_0)$,
the following quasi-Clairaut formula holds, for all $t_1,t_2\in[0,T]$:
\[
C^{-1}\delta(\gamma_v(t_2))^rb(\dot\gamma_v(t_2))\leq \delta(\gamma_v(t_1))^rb(\dot \gamma_v(t_1))\leq C\delta(\gamma_v(t_2))^rb(\dot\gamma_v(t_2)).\]
\end{proposition}

\begin{proof}[Proof of Proposition~\ref{p=Clairaut}]
First note that the statement is trivially true if $b(\dot\gamma_v(0))=0$; hence we may assume $b(\dot\gamma_v(t))\neq 0$.
Let $\gamma_v\colon[t_1,t_2]\to \DD^\ast(\delta_0)$, where $\delta_0$ will be specified later.

As in the previous section, write  
$$
a(t) = a(\dot\gamma_v(t)), b(t) =  b(\dot\gamma_v(t)), c(t) =  c(\gamma_v(t)),\;\hbox{ and } 
\delta(t) =  \delta(\gamma_v(t)).
$$

The first main ingredient in the proof of Proposition~\ref{p=Clairaut} is the following lemma.

\begin{lemma}\label{l=convexd} 
For every $v\in T^1\DD^*$, the function $\delta$ is convex along $\gamma_v$ and strictly convex if  $a(v)\notin \{-1,+1\}$, 
\end{lemma}

\begin{proof} 
At any time $t$  where $\delta'(t)=a(t)=0$ we have $b(t)=1$ so that by Lemma~\ref{l=derivs}, we have
 $\delta''(t)=a'(t)= b(t)^2 c(t) = c(t)$, which is positive,
by Proposition~\ref{p=c}.
\end{proof}

Returning to the proof of Proposition~\ref{p=Clairaut}, let $g=\delta^rb$. We first calculate:
$$g' = (\delta^rb)'=r\delta^{r-1}b\delta'+\delta^rb'=r\delta^{r-1}ab+\delta^r\frac{-rab}{\delta}+O(|a|b\delta^{r}),$$
and so  $g'/g=O(|a|)$.  Thus there is a constant $C$ such that 
$\left|{g'}/{g}\right|\leq C|a|=C|\delta'|$.
Fixing $t_1 < t_2$, we have
 $$\left |\int_{t_1}^{t_2}  \frac{g'}{g} \right|\leq \int_{t_1}^{t_2} \left |\frac{g'}{g } \right|\leq \int_{t_1}^{t_2} C|\delta'|.$$  
Thus
$$\exp\left({-C\int_{t_1}^{t_2} |\delta'|}\right)\leq \frac{g(t_1)}{g(t_2)}\leq \exp\left({C\int_{t_1}^{t_2}|\delta'|}\right).$$

Since $\delta$ is convex and $\delta(t_1), \delta(t_2)\leq \delta_0$, we have $\delta(t)\leq \delta_0$ for all $t\in [t_1,t_2]$,
and there is at most one $t_*\in [t_1,t_2]$ where $\delta'$ vanishes. It follows easily that
$\int_{t_1}^{t_2} |\delta'| \leq  2 \delta_0$,
which implies the conclusion.
\end{proof}

\begin{corollary}\label{p=leave} Every unit-speed geodesic in $\DD^\ast$ that enters the region $\DD^\ast(\delta_0)$  leaves the
region in time $\leq 2\delta_0$.
\end{corollary}

\subsection{Cuspidal Jacobi fields}\label{s=jacobi}

For $\gamma$ a geodesic segment in $\DD^\ast$, we consider solutions to the Riccati equation:
\begin{equation}\label{e=zeta}
\zeta'(t) = 1+K(\gamma(t))  \zeta(t)^2,
\end{equation}
which is defined on a time interval containing $0$.  The next lemma  shows that there is a ``cone condition" on initial data that is preserved by solutions to (\ref{e=zeta}).  We use this in the next section to construct an invariant cone field  for solutions to  (\ref{e=zeta}) in $S$.

\begin{lemma}\label{l=cone} For every $\epsilon >0 $ there exists $\delta_0>0$  such that
the following holds for every geodesic segment $\gamma\colon[0,T]\to \DD^\ast(\delta_0)$.
Let  $\zeta$ be a solution to  (\ref{e=zeta}) along $\gamma$.
\begin{enumerate}
\item  If $\zeta(0) \leq  \delta(v)/(r-1-\epsilon )$, then
$\zeta(t) \leq \delta(t)/(r-1-\epsilon)$ for all $t\in [0, T]$, and
\item if $\zeta(0) \geq \delta(v)/(r+\epsilon)$, 
 then $\zeta(t)\geq  \delta(t)/(r+\epsilon)$, for all $t\in [0, T]$.
\end{enumerate}
\end{lemma}

\begin{proof} We establish the lower bound first. To this end let $w=\delta/(r+\epsilon)$. Then
$w' = a/(r+\epsilon)$.  To show that $w\leq \zeta$, it suffices by Lemma~\ref{l=comparison} to show that
$$
 \frac{a}{r+\epsilon} \leq  1 + K w^2 = 1 + \left( - \frac{r(r-1)}{\delta^2} + O(\delta^{-1}) \right)\frac{\delta^2}{(r+\epsilon)^2};
$$
equivalently, it suffices to show that
$ r - ar +2r \epsilon - a\epsilon + \epsilon^2 \geq O(\delta)$.
Since $a\leq 1$, this clearly will hold if $\delta_0$ is sufficiently small.
The upper bound is proved similarly.
\end{proof}

\section{Global properties of the flow in $T^1 S$}\label{s=global}

Now consider the surface $S$ with one puncture, satisying the hypotheses of Theorem~\ref{t=main}.  Let $\delta$ be the distance to the cusp.
For $\delta_0>0$, denote by  $ \cN(\delta_0) = \{p: \delta(p)\leq \delta_0\}$ the convex
$\delta_0$-neighborhood of the cusp.  In this section, we modify the function $\delta$ outside of a neighborhood  $\cN(\delta_1)$ and use the modified function $\bar\delta$ to construct a $D\varphi_t$-invariant conefield on $T T^1S$.  We also use the modified function $\bar\delta$ to construct a new Riemannian  metric on $T^1S$,
called the $\star$ metric, that makes $T^1S$ complete.  

Having done this, we consider the flow $\psi_t$  on $T^1S$ given by rescaling $\varphi_t$ to have unit speed in the $\star$ metric.  We prove that this flow is Anosov in the $\star$ metric and preserves a smooth, finite volume.  This allows us to conclude that $\varphi_t$ is ergodic and has smooth invariant stable and unstable foliations on which $\varphi_t$ acts with bounded distortion.

\subsection{Invariant cone field}
In this subsection, we prove the following key technical result, which we will use to prove that a rescaled version of $\varphi_t$ is Anosov.
\begin{proposition}\label{p=cones4}{\bf[Cones]} For every $\epsilon>0$ sufficiently small, if $\delta_1$ is sufficiently small, then the following holds.

 There exists $\beta>0$, an extension $\bar\delta:S\to(0,\infty)$ of $\delta|_{\cN(\delta_1/2)}$ and a function  $\chi\colon S\to [\beta, r-1-\epsilon]$ satisfying $\chi(p) = r-1-\epsilon$ for $\bar\delta(p)  \leq \delta_1$,
\[\chi(p) - \|\nabla \bar\delta(p)\| \geq  \beta,\]
for all $p\in S$,
 and  such that the following holds.  Let $\gamma\colon [0,T]\to  S$ be a geodesic segment in $S$.  Then:
\begin{enumerate}
\item If  $u$ is a solution to  (\ref{e=Ricplain}) below then
\[ u(0) \in \left[\frac{\chi(\gamma(0))}{\bar\delta(\gamma(0))}, \frac{r+\epsilon}{\bar\delta(\gamma(0))}\right]\implies u(t) \in \left[\frac{\chi(\gamma(t))}{\bar\delta(\gamma(t))}, \frac{r+\epsilon}{\bar\delta(\gamma(t))}\right],\]
for all $t\in [0,T]$,
\item  If  $u$ is a solution to (\ref{e=Ricplain}) below then
\[ u(T) \in \left[- \frac{r+\epsilon}{\bar\delta(\gamma(T))}, -\frac{\chi(\gamma(T))}{\bar\delta(\gamma(T))}\right] \implies u(t) \in  \left[- \frac{r+\epsilon}{\bar\delta(\gamma(r))}, -\frac{\chi(\gamma(t))}{\bar\delta(\gamma(t))}\right], \]
for all $t\in [0,T]$.
\end{enumerate}
\end{proposition}
\begin{proof}
The proof is broken into a few steps. 
\subsubsection{The lower edge of the cone: $j'/j\geq g$.}
\begin{lemma}\label{l=cone2}  For every $\epsilon>0$, there exist $\mu \in (0,1) $ and for every
$\delta_0>0$ sufficiently small, a continuous function  $g \colon S\to (0,\infty)$
with the following properties:
\begin{enumerate}
\item For every $p\in S$, we have
$g(p) \leq  {(r- 1- \epsilon)}/{ \delta(p) }$.
\item For every $p\in \cN(\mu\delta_0)$, we have
$g (p) =   {(r- 1- \epsilon)}/{ \delta(p) }$.
\item
Let $\gamma\colon [0,T]\to S$ be a geodesic segment in $ S$, and let  $u$ be any solution to  
\begin{equation}\label{e=Ricplain} u'(t) = -K(\gamma(t))-u(t)^2.
\end{equation}
  Suppose that
$u(0) \geq g(\gamma(0))$.
Then $u(t) \geq g(\gamma(t))$,
for all $t\in[0,T]$.

\end{enumerate}

\end{lemma}

\begin{proof}

Given $\epsilon < (r-2)/2$, we choose 
$\delta_0\in(0,1)$ sufficiently small according to  Lemma~\ref{l=cone}. 
Let $-\kappa_0^2$ be an upper bound on the curvature on $S$, and let
\[
\theta = \min\left\{\kappa_0, \inf_{p\in S}\frac{r-1-\epsilon}{\delta(p)}\right\}.
\]

We fix $\mu = \mu(\epsilon) >0$ very small (to be specified later).
Let $\eta\colon [\mu\delta_0,\delta_0]\to \RR_{>0}$ be the  affine function satisfying
\[\eta(\mu\delta_0) = r-1-\epsilon,\,\hbox{ and }\eta(\delta_0) = \theta\delta_0,\]
and define $g\colon S\to (0,\infty)$ by:
\[g(p) = \begin{cases} 
 \theta &\mbox{ if } p\in S\setminus \cN(\delta_0), \\ 
 \frac{\eta(\delta(p))}{\delta(p)}& \mbox{ if } p\in \cN(\delta_0) \setminus \cN(\mu\delta_0), 
\\ 
 \frac{r- 1- \epsilon}{ \delta(p) } &\mbox{ if } p\in \cN(\mu\delta_0). \\ 
 \end{cases} 
\]
By construction, $g$ satisfy conditions 1 and 2.
We check invariance of the condition $u(t)\geq g(\gamma(t))$;  to this end, let $\gamma\colon [0,T]\to S$ be a geodesic, and suppose that $u$ is a solution to (\ref{e=Ricplain}) satisfying $u(0)\geq g(\gamma(0))$.
By breaking $\gamma$ into pieces if necessary, we may assume that one of the following holds:
\begin{enumerate}
\item[] {\bf Case 1.} $\gamma[0,T]\subset S\setminus \cN(\delta_0)$,
\item[] {\bf Case 2.} $\gamma[0,T]\subset \cN(\delta_0) \setminus \cN(\mu\delta_0)$, or
\item[] {\bf Case 3.} $\gamma[0,T]\subset \cN(\mu\delta_0)$.
\end{enumerate}

 Cases 1 and 3 are pretty trivial.  In Case 1, $g\equiv \theta$, and the fact that $-K\geq \kappa_0^2 \geq \theta^2$ implies that the condition $u\geq \theta$ is invariant. In Case 3, $g = (r-1-\epsilon)/\delta$, and we apply Lemma~\ref{l=cone}.

In Case 2,  we will apply Lemma~\ref{l=comparison} to the function $u_0(t):= g(\gamma(t))$.
Differentiating $u_0$, we have
\[
u_0'(t) = \left(\frac{\eta(\delta(\gamma(t)))}{\delta(t)}\right)' = \frac{\eta'(\delta(\gamma(t))) a(\gamma(t))}{\delta(\gamma(t))} - \frac{\eta(\delta(\gamma(t)))a(\gamma(t))}{\delta(\gamma(t))^2}.
\]
Lemma~\ref{l=comparison} implies that $u(t)\geq u_0(t)$ for all $t\in[0,T]$ provided that
$u(0)\geq u_0(0)$ and $-K(\gamma(t)) - u_0(t)^2 \geq u_0'(t)$, for all $t$.  The latter is equivalent to:
\begin{equation}\label{e=u0comp} -K(\gamma(t)) -\left( \frac{\eta(\delta(\gamma(t)))}{\delta(\gamma(t))} \right)^2\geq \frac{\eta'(\delta(\gamma(t))) a(\gamma(t))}{\delta(\gamma(t))} - \frac{\eta(\delta(\gamma(t)))a(\gamma(t))}{\delta(\gamma(t))^2}.
\end{equation}
Since $K = - r(r-1)/\delta^2 + O(1/\delta)$, if $\epsilon$ and $\delta_0$ are sufficiently small, inequality (\ref{e=u0comp}) will hold provided that
\begin{equation}\label{e=etacomp}
r(r-1-\epsilon) - \eta^2  \geq a\eta' \delta - a \eta.
\end{equation}
Since $\eta \in (0, r-1-\epsilon]$ and $\eta' < 0$, inequality (\ref{e=etacomp}) holds automatically when $a\geq 0$.  For $a\leq 0$, inequality  (\ref{e=etacomp}) will hold provided that for all $\delta\in [\mu\delta_0,\delta_0]$, we have:
\begin{equation}\label{e=etacomp2}
r(r-1-\epsilon) - \eta^2  - \eta  \geq -\eta' \delta.
\end{equation}
Since
$-\eta' \leq   {(r-1-\epsilon)}/{((1-\mu)\delta_0)}$,
we are reduced to proving the inequality
\[
r(r-1-\epsilon) - \eta^2  - \eta  \geq   \frac{(r-1-\epsilon)\delta} {(1-\mu)\delta_0}.
\]
To verify this, it suffices to show that the correct inequality holds at the endpoints
$\delta = \mu\delta_0$ and $\delta = \delta_0$; this is easily verified  provided $\delta_0$ and $\mu = \mu(\epsilon)$
are sufficiently small.\end{proof}

\subsubsection{Definition of the modified distance function
 $\bar\delta$.}\label{ss=bardelta}

Fix $\epsilon < (r-2)/2$,  and let $\delta_0>0$  and $\mu\in(0,1)$ be given by Lemma~\ref{l=cone2}.  Let  $\delta_1 = \mu\delta_0$.
Since $S\setminus\cN(\delta_1)$ is compact, we may assume that $\delta_0$ (and hence $\delta_1$) is small enough that 
\[
-\kappa_1^2(\delta_1):= \inf_{p\in S\setminus\cN(\delta_1)} K(p) >  - \frac{(r+\epsilon)(r-1+\epsilon) }{\delta_1^2}.
\]
Fix $\lambda\in (0,1)$ close enough to $1$ that 
\[
-\kappa_1^2(\delta_1) >  - \frac{(r+\epsilon)(r-1+\epsilon) }{(\lambda\delta_1)^2},
\]
and
$\beta_0:= \lambda(r-1-\epsilon) -1  > 0$.

 We extend $\delta$ to  a $C^4$ function $\bar\delta\colon  S\to \RR_{>0}$
satisfying $\bar\delta(p) = \delta(p)$ for $p\in \cN(\delta_1/2)$ and $\bar\delta(p) =\lambda\delta_1$,
for $p\in   S \setminus \cN(\delta_1)$.  We may do this so that
$\bar\delta/\delta \geq \lambda$, and $\|\nabla\bar\delta\|\leq 1$ 
in $\cN(\delta_1)$.
 We also denote by $\bar\delta$ the function on $T^1 S$ defined by
$\bar\delta(v) = \bar\delta(\pi(v))$,
which is constant on the fibers of $T^1 S$.  Thus if $v\in T^1 S$, and $t\in \RR$, we have:
\[
\bar\delta(\varphi_t(v)) = \bar\delta(\gamma_v(t)),
\]
and we will at times write these expressions interchangeably.

Let $g\colon S\to \RR_{>0}$ be the lower cone function given by Lemma~\ref{l=cone2}.  Define
$\chi\colon S \to \RR_{>0}$ by $\chi(p) = \bar\delta(p) g(p)$.  As with
$\bar\delta$, we will lift $\chi$ to a function on $T^1S$ and write $\chi(v) = \chi(\pi(v))$.
Our choice of $\lambda$ ensures that the following lemma holds

\begin{lemma}\label{l=chi}  For all $p\in S$, we have $\chi(p) \leq r-1-\epsilon$.
There exists $\beta >0$ such that for all $p\in S$, we have
$\chi(p) - \|\nabla \bar\delta(p)\| \geq  \beta$.
\end{lemma}
\begin{proof} 
The first assertion follows easily from the fact that
$\bar\delta/\delta \leq 1$ and part 1 of Lemma~\ref{l=cone2}.

 If $p\in S\setminus \cN(\delta_1)$, then $\nabla\bar\delta(p)=0$, and the conclusion holds with $\beta_1 = \inf_{S\setminus \cN(\delta_1)} \chi > 0$.  If $p\in  \cN(\delta_1)$, then $\chi(p) = (r-1-\epsilon)\bar\delta(p)/\delta(p)$,  $\bar\delta(p)/\delta(p) \geq \lambda$ and $\|\nabla\bar\delta(p)\|\leq 1$; thus
$\chi(p) - \|\nabla \bar\delta(p)\| \geq  \lambda(r -1-\epsilon) - 1 = \beta_0>0$.
We conclude by setting $\beta = \min\{\beta_0,\beta_1\}$.
\end{proof}

\subsubsection{The upper edge of the cone: $j'/j\leq (r+\epsilon)/\bar\delta$.}

Using the modified cuspidal distance function $\bar\delta$, we now can define an upper edge to an invariant cone field for solutions to (\ref{e=Ricplain}).
\begin{lemma}\label{l=cone3} Let $\bar\delta$ be defined as in Section~\ref{ss=bardelta}.
Let $\gamma\colon [0,T]\to S$ be a geodesic segment in $S$, and let  $u$ be any solution to  (\ref{e=Ricplain}) with
$u(0) \leq (r+\epsilon)/\bar\delta(\gamma(0))$.
Then $u(t) \leq (r+\epsilon)/\bar\delta(\gamma(t))$,
for all $t\in[0,T]$.
\end{lemma}
\begin{proof} This is a straightforward application of Lemma~\ref{l=comparison}, using only the facts  that $\|\nabla \bar\delta\|\leq 1$, and
$-K(p) \leq {(r+\epsilon)(r-1+\epsilon) }/{\bar\delta(p)^2}$,
for all $p\in S$.
\end{proof}

Lemmas~\ref{l=cone2}, \ref{l=chi}  and \ref{l=cone3} can be applied as well to the flow $\varphi_{-t}$ to obtain invariant negative cones for solutions to the equation (\ref{e=Ricplain}). 
One can do this using the same functions $\chi,\bar\delta$ satisfying both (1) and (2) in the conclusion of Proposition~\ref{p=cones4}
This completes the proof of the proposition.  \end{proof}

\subsection{An adapted, complete metric on $T^1 S$}\label{s=adapted}

Define a new Riemannian metric on $T^1 S$ by
\[\langle (w_1, w_1'), (w_2, w_2') \rangle_{\star,v} =  \frac{1}{\bar\delta(v)^2}\langle w_1, w_2 \rangle_{\pi(v)} +
{\langle w_1', w_2' \rangle_{\pi(v)}},\]
for $v\in T^1S$.

\begin{remark} The $\star$ metric on $T^1S$ is comparable (i.e. bi-Lipschitz equivalent) to the induced Sasaki metric for the so-called Ricci metric on 
$S$.  (The Ricci metric  on $S$ is obtained by scaling the original metric by $-K$.)  We briefly explain.

  Define a metric $\langle\cdot,\cdot\rangle_\dagger$ on $S$ by conformally rescaling the original metric, as follows:
\[
\langle\cdot,\cdot\rangle_\dagger = \bar\delta^{-2}  \langle\cdot,\cdot\rangle.
\]
This is comparable to the Ricci metric, since $-K$ is comparable to $\bar\delta^{-2}$.

We claim that the metric on $T^1S$ induced by the Sasaki metric for $\langle\cdot,\cdot\rangle_\dagger$  is comparable
to  $\langle\cdot,\cdot\rangle_\star$.
Here is a crude sketch of the proof.
The unit tangent bundle $T^1S$ for the original metric is clearly not the $\langle\cdot,\cdot\rangle_\dagger$ unit tangent bundle, but angles remain the same, and so $\langle\cdot,\cdot\rangle$ angular distance in
the vertical fibers of $T^1S$ coincides with $\langle\cdot,\cdot\rangle_\dagger$ 
angular distance.  On the other hand, $\dagger$-distance in the horizontal fibers of $TS$ (with respect to the original connection) is the original distance scaled by $\bar\delta^{-1}$.  Thus the formulas are comparable.

As it is more convenient to work with the $\star$ metric, we will not pursue here further the properties of the $\dagger$ metric on $S$, but one can prove that (for $\delta_0$ sufficiently small) it is complete,  negatively curved with pinched curvature, and of finite volume.  In the case where the original metric is the WP metric, the $\dagger$ metric is comparable to the Teichm\"uller metric, which is the hyperbolic metric.  We will not be using the Riemannian properties of the $\star$ metric beyond completeness and finite volume.
\end{remark}

\medskip

Let $\rho_\star$ be the Riemannian distance on $T^1 S$ induced by $\langle\;,\,\rangle_\star$.
\begin{lemma} 
 $\rho_\star$ is complete.
\end{lemma}

\begin{proof} By the Hopf-Rinow theorem, it suffices to show that  any  
$\star$-geodesic  is defined for all time. The only way in which a geodesic in $T^1S$ can stop being defined is for its projection to 
$S$ to hit the cusp. But the projection to $S$ of a $\star$-geodesic is a curve that has speed $\delta$ when it is at distance $\delta$ from the cusp in the geometry of our Riemannian metric $\langle \cdot,\cdot \rangle$ on $S$. It is clear that such a curve cannot reach the cusp in finite time.
\end{proof}

\subsection{Lie brackets and $\star$-covariant differentiation on $T^1S$}
If $X$ is a vector field on $S$, then $X$ has two well-defined lifts $X^h$ and $X^v$ to vector fields on
$TS$, the {\em horizontal} and {\em vertical} lifts, respectively. They are defined by
\[
X^h(u) = (X(\pi(u)),0),\quad\hbox{ and } X^v(u) = (0, X(\pi(u))),
\]
for $u\in T^1S$.
The following formulas for Lie brackets of such lifts are standard; see \cite{Ko}.
\begin{lemma} Let $X$ and $Y$ be arbitrary vector fields on $S$. Then
\begin{itemize} 
\item $[X^v, Y^v]_u = 0$
\item $[X^h, Y^v]_u = (0, \nabla_XY)$
\item $[X^h, Y^h]_u = ([X,Y], - R(X,Y)u)$
\end{itemize}
\end{lemma}

Recall the definitions of $\bar V = \nabla \bar\delta$ and $J\bar V$.  To simplify notation, and since the calculations that follow are only interesting in the thin part $\cN(\delta_1/2)$ where
$\delta = \bar\delta$,  we will write
$V, JV$, and $\delta$ for their barred counterparts in what follows.

\begin{lemma}\label{l=brackets} Denote by $V^h, JV^h, V^v, JV^v$ the horizontal and vertical lifts, respectively, of $V$ and $JV$.  Then for $u\in \cN(\delta_1/2)$, we have:
\begin{enumerate}
\item $[V^v, JV^v]_u = [V^h, JV^v]_u = [V^h, V^v]_u = 0$,
\item $[JV^h, V^v]_u = (0, cJV)$
\item $[JV^h, V^h]_u = (cJV,  -R(JV,V)u)$
\item $[JV^h, JV^v]_u = (0, -cV)$
\end{enumerate}
\end{lemma}
\begin{proof} This is a direct application of the previous lemma and the fact that
 $\nabla_{JV} V = [JV,V] = cJV$ from Proposition~\ref{p=c}.
\end{proof}

Observe that $\|V^h\|_\star = \|JV^h\|_\star = \delta^{-1}$, and
 $\|V^v\|_\star = \|JV^v\|_\star = 1$.  We have:
\begin{lemma}  Let $X$ and $Y$ be arbitrary vector fields on $S$ with $\|X\|=\|Y\| = 1$,
and denote by $X^h, X^v, Y^h, Y^v$ their horizontal and vertical lifts.  Then
\[
\|\nabla^\star_{X^h} Y^h\|_\star = O(\delta^{-2}), \, \|\nabla^\star_{X^h} Y^v\|_\star = O(\delta^{-1}), \, \|\nabla^\star_{X^v} Y^h\|_\star = O(\delta^{-1}),\]
and
\[ 
\nabla^\star_{X^v} Y^v = 0.
\]
In particular, the $\star$ connection is summarized in  Table 1.
\end{lemma}

\begin{table}[ht]
\resizebox{\textwidth}{!}{%
\begingroup
\setlength{\tabcolsep}{2pt} 
\renewcommand{\arraystretch}{1.6} 
\begin{tabular}{ r|c|c|c|c| }
\multicolumn{1}{r}{}
 &  \multicolumn{1}{c}{$V^h$}
 & \multicolumn{1}{c}{$JV^h$} 
&  \multicolumn{1}{c}{$V^v$}
 & \multicolumn{1}{c}{$JV^v$} 
\\
\cline{2-5}
$V^h$ &  & $-\delta^{-1} JV^h $ &  & \\
 & $- \delta^{-1} V^h$ & $+ \frac{1}{2} \langle R(JV,V)u) ,V\rangle  V^v  $ &$ - \frac{\delta^2}{2} \langle R(JV,V)u), V\rangle JV^h$ &$ - \frac{\delta^2}{2} \langle R(JV,V)u, JV\rangle JV^h$  \\
 &  & $ + \frac{1}{2} \langle R(JV,V)u, JV\rangle JV^v$ & & \\
\cline{2-5}
$JV^h$  &   $(- \delta^{-1} + c)JV^h $  & &  & \\
 &  $ -\frac{1}{2} \langle R(JV,V)u, V \rangle V^v$ &  $( \delta^{-1} -  c ) V^h  $  & $\frac{\delta^2}{2}\langle  R(JV,V)u),V\rangle V^h$ & $\frac{\delta^2}{2}\langle R(JV,V)u),JV\rangle V^h$\\
 &  $ -\frac{1}{2} \langle R(JV,V)u, JV\rangle JV^v$ &  & $+c JV^v $& $-cV^v$\\
\cline{2-5}
$V^v$ & $ -\frac{\delta^2} 2  \langle R(JV,V)u,V\rangle JV^h$ & $ \frac{\delta^2}{2}\langle R(JV,V)u),V\rangle V^h$ & $0$ & $0$ \\
\cline{2-5}
$JV^v$ & $-\frac{\delta^2}{2}   \langle  R(JV,V)u), JV \rangle JV^h$ & $\frac{\delta^2}{2}   \langle  R(JV,V)u), JV \rangle V^h$ & $0$ & $0$ \\
\cline{2-5}
\end{tabular}
\endgroup}
\end{table}
\noindent \centerline{{\bf Table 1:} $\nabla^\star_X Y(u)$, for $u\in T^1\cN(\delta_1/2)$, where $X$ is the row vector field,} \\ \centerline{and $Y$ is the column vector field.}

\bigskip\medskip

\begin{proof} The proof is a calculation using Lemma~\ref{l=brackets} and
Koszul's formula:
\[\begin{array}{cc}2 \langle \nabla_X^\star Y, Z\rangle_\star &=\, X (\langle Y,Z\rangle_\star) + Y (\langle X,Z\rangle_\star) - Z (\langle X,Y\rangle_\star) \\&\quad+\, \langle[X,Y],Z\rangle_\star - \langle [X,Z],Y\rangle_\star - \langle[Y,Z],X\rangle_\star.\end{array}\]
The details can be found in \cite{preprint}.
\end{proof}

\begin{lemma} Let $a(w) = \langle w, V(\pi(w))\rangle, b(w) =\langle w, JV(\pi(w))\rangle$ be defined as above.  Then
\begin{enumerate}
\item  $V^h(a) = V^h(b) = 0$,
\item $JV^h(a) = bc$, and $JV^h(b) = -ac$,
\item $V^v(a) =1$, and  $V^v(b) =0$
\item $JV^v(a) =0$, and  $JV^v(b) =1$
\end{enumerate}
\end{lemma}
\begin{proof} \ 1. \
To differentiate a function $\phi$ on $T^1S$ with respect to $V^h$ at $w\in T^1S$, we parallel translate $w$ along the geodesic $\gamma(t)$
through $\pi(w)$ tangent to $V$ to obtain  $\Pi_t^V(w)$,  and then differentiate the function $\phi(\Pi^V_t(w))$ with respect to $t$ at $t=0$.   Since $\gamma$ is a geodesic tangent to $V$, the angle between $\Pi^V_t(w)$ and $V$ remains constant, and so $a(\Pi_t^V(w))$ and $b(\Pi_t^V(w))$ are both constant.  Thus their derivatives are both zero.

2.\ 
Let $\Pi_t^{JV}(w)$ be the parallel translate of $w$ along the integral curve of the vector field $JV$ through $\pi(w)$. Then

\[JV^h(a)(w) =  \frac{d}{dt} \langle \Pi_t^{JV}(w), V\rangle \vert_{t=0} = \langle w,  \nabla_{JV} V\rangle = 
\langle w,  c JV\rangle = bc,
\]
and
\[JV^h(b)(w) =  \frac{d}{dt} \langle \Pi_t^{JV}(w), JV\rangle \vert_{t=0} = \langle w,  \nabla_{JV} JV\rangle = 
\langle w,  -c V\rangle = -ac.
\]

3.\ 
To compute the derivative $V^v\phi$ at $w$, we differentiate $\phi(w+tV)$ at $t=0$ in the fiber over $\pi(w)$.
Thus
\[V^v(a)(w) =  \frac{d}{dt} \langle w + tV, V\rangle \vert_{t=0}  = \frac{d}{dt} t \langle V, V\rangle \vert_{t=0} = 1,
\]
and
\[V^v(b)(w) =  \frac{d}{dt} \langle w + tV, JV\rangle \vert_{t=0}  = \frac{d}{dt} t \langle V, JV\rangle \vert_{t=0} = 0.
\]

4.\ 
To compute the derivative $JV^v\phi$ at $w$, we differentiate $\phi(w+tJV)$ at $t=0$ in the fiber over $\pi(w)$. The calculations are similar to those in 3.
\end{proof}

\bigskip

\begin{proposition}\label{p=starnabla}  Let $X$ be any  vector field on  $T^1S$ with $\|X\|_\star = 1$.
Then $$\|\nabla^{\star}_X \dot\varphi\|_\star = O(\delta^{-1}).$$ In particular:
\begin{enumerate}
\item $\nabla^\star_{V^h} \dot \varphi =  -a \delta^{-1} V^h - b \delta^{-1} JV^h - \frac{b^2 K}{2} V^v + \frac{ab K}{2} JV^v $
\item $\nabla^\star_{JV^h} \dot \varphi = b\delta^{-1}V^h - a\delta^{-1} JV^h +  \frac12 Kab V^v - \frac12 Ka^2JV^v$
\item $\nabla^\star_{V^v} \dot \varphi =  \frac{K ab\delta^2} {2}   JV^h+ \left(\frac{-Kb^2\delta^2}{2} + 1\right) V^h $
\item $\nabla^\star_{JV^v} \dot \varphi = \left(1-\frac{Ka^2\delta^2}{2}\right) JV^h +  \frac{K ab \delta^2}{2} V^h $
\end{enumerate}
\end{proposition}

\begin{proof} The proof is just a calculation.  To see 1, for example, observe that
\[
 \nabla^\star_{V^h} \dot\varphi =
  \nabla^\star_{V^h}  (aV^h + bJV^h)
=  a  \nabla^\star_{V^h} V_h + b \nabla^\star_{V^h} JV_h  \]
\[= -a \delta^{-1} V^h - b  \delta^{-1} JV^h + \frac{b}{2} \langle R(JV,V)u ,V\rangle  V^v + \frac{b}{2} \langle R(JV,V)u ,JV\rangle JV^v,
\]
where $u=aV + bJV$.
Thus
\[\nabla^\star_{V^h} \dot\varphi =  -a \delta^{-1} V^h - b \delta^{-1} JV^h - \frac{b^2 K}{2} V^v + \frac{ab K}{2} JV^v .
\]

The other formulas are proved similarly; see \cite{preprint} for the details.
\end{proof}

\subsection{Time change to an Anosov flow}

As above, let $\dot\varphi$ be the geodesic spray; i.e. the generator of the geodesic flow on $T^1 S$.
Define a new flow $\psi_t$ on $T^1 S$ with generator
$$
\dot\psi(v) = \bar\delta(v) \dot\varphi(v).
$$
One might ask first whether this flow is complete; that is, is it defined for all time $t\in\RR$, for each $v\in T^1 S$?  Note that the original flow $\varphi_t$ is not complete, since it is the geodesic flow of an incomplete manifold.  The completeness of $\psi_t$ follows from the completeness of $T^1 S$ in the 
$\star$-metric defined above, and the following lemma.

\begin{lemma} The vector field $\dot\psi$ is $C^3$, and there exists a constant $C>0$ such that for every
$v\in T^1 S$, 
$$\|\dot\psi\|_\star =  1, \quad\hbox{and } \, \|{\nabla^{\star}} ^i\dot\psi\|_\star \leq C, \quad \hbox{ for } i=1,2,3.
$$

The flow $\dot\psi_t$ preserves a finite measure $\mu$ on $T^1\Sigma$ that is equivalent to Liouville volume
for the original metric:
$d\mu =  { \bar\delta }^{-1} {d\vol}$.

\end{lemma}

\begin{proof}  By definition of the $\star$ metric, we have $\|\dot\psi\|_\star = \|\bar\delta\dot\varphi\|_\star= 1$.

Since $\bar\delta(v) = \bar\delta(\pi(v))$, the derivatives of $\bar\delta$ have no vertical component, and
 Corollary~\ref{l=cuspdist} gives that $\|\nabla^i \bar\delta\| = O(\bar\delta^{1-i})$.  A unit horizonal
vector in the $\star$ norm is of the form $(\xi_H, 0)$, where $\|\xi_H\| = \bar\delta$.  Thus 
\begin{equation}\label{e=diffdelta}
\|{\nabla^\star}^i\bar\delta\|_\star = \bar\delta^i \|{\nabla}^i \bar\delta\| = O(\bar\delta),
\end{equation}
for $i=1,2,3$.

Proposition~\ref{p=starnabla} implies $\|\nabla^\star\dot\varphi\|_\star$ has magnitude $\bar\delta(v)^{-1}$ in the  $\star$ metric.  A similar calculation taking higher covariant derivatives of the formulas in Proposition~\ref{p=starnabla}  and using the facts that
$\|\nabla K\| = O(\bar\delta^{-3})$,  $\|\nabla^2 K\| = O(\bar\delta^{-4})$,
$\|\nabla^i c\| = \delta^{-1-i}$, and  (\ref{e=diffdelta})
gives that 
\begin{equation}\label{e=diffdotvarphi}
\|{\nabla^\star}^i\dot\varphi\|_\star = O(\bar\delta^{-1}),
\end{equation} for $i=1,2,3$.

  Combining (\ref{e=diffdotvarphi}) and (\ref{e=diffdelta}), we obtain that
$$\|\nabla^\star(\bar\delta\dot\varphi)\|_\star = \| \nabla^\star (\bar\delta) \dot\varphi \|_\star  + \|\bar\delta \nabla^\star \dot\varphi\|_\star
\leq  \| \nabla^\star (\bar\delta) \|_\star \| \dot\varphi \|_\star  + \bar\delta \| \nabla^\star \dot\varphi\|_\star = O(1).
$$
Similarly, we obtain that $\|{\nabla^\star}^i(\bar\delta\dot\varphi)\|_\star  = O(1)$, for $i= 2,3$.

Let $\omega$ be the canonical one form on the tangent bundle $T S$ with respect to the
original metric.  Then $\varphi^\ast_t \omega = \omega$, for all $t$, and
$d\vol = \omega\wedge d\omega$ on $T^1S$.   We have that: 
$$
\cL_{\dot\psi}\left( \bar\delta^{-1} \omega\wedge d\omega\right)  = d\left(\iota_{\dot\psi} \left(\bar\delta^{-1} \omega\wedge d\omega \right)\right) = d (d\omega) = 0,
$$
since $\bar\delta^{-1} \omega(\dot\psi) = \omega(\dot\varphi) \equiv 1$.  Thus
$\psi_t$ preserves the smooth measure $\mu$ defined by $d\mu = \bar\delta^{-1} \omega\wedge d\omega = \bar\delta^{-1} d\vol$.

To see that $\mu(T^1S)<\infty$, we use the expression for $d\vol$ from  (\ref{e=volume}) and integrate:
\[
\mu(T^1S) = \int_{T^1S}  \bar\delta^{-1} d\vol = O\left( \int_{0}^{\delta_0} x^{r-1} \,{dx} \right)  < \infty.
\]\end{proof}

The flow $\psi_t$ is a time change of $\varphi_t$; that is, it has the same orbits, but they are traversed at a different speed, depending on the distance to the singular locus.  Indeed, defining the cocycle $\tau\colon T^1 S\times \RR\to \RR$ by the implicit  formula
\begin{equation}\label{e=taudef}
\int_0^{\tau(v,t)} \frac{dx}{\bar\delta(\varphi_x(v))} = t,
\end{equation}
we have that $\psi_{t}(v) = \varphi_{\tau(v,t)}(v)$,
for all $v\in T^1 S$,  $t\in\RR$.  This gives an alternate way to see the completeness of the flow $\psi$: the
function $\bar\delta$ clearly remains positive along orbits of $\psi$ for all time.

\begin{theorem}\label{t=accelanosov}
The flow $\psi_t$ is an Anosov flow in the $\star$-metric.  That is, there exists a $D\psi_t$-invariant, continuous splitting
of the tangent bundle:
$$
T\left(T^1 S\right) = E^u_\psi\oplus \RR\dot\psi \oplus E^s_\psi
$$
and constants $C>0$,  $\lambda > 1$ such that for every $v\in T^1 S$, and every $t>0$:
\begin{itemize}
\item $\xi\in E^u_\psi(v) \implies \|D\psi_{-t} (\xi)\|_\star \leq C\lambda^{-t} \|\xi\|_\star$, and
\item  $\xi\in E^s_\psi(v) \implies \|D\psi_t (\xi)\|_\star \leq C\lambda^{-t} \|\xi\|_\star$.
\end{itemize}
\end{theorem}

From Theorem~\ref{t=accelanosov} we obtain several important properties of both $\psi_t$ and $\varphi_t$.
The first is ergodicity.
Since volume preserving Anosov flows are ergodic, the flows $\varphi_t$ and $\psi_t$ have the same orbits, 
and the $\star$ volume is equivalent to (i.e. has the same zero sets as)  the original volume on $T^1S$, we obtain:
\begin{corollary}\label{c=ergodic}
The flow $\psi_t$ is ergodic with respect to the invariant volume $\mu$.  Consequently, $\varphi_t$ is ergodic with respect to volume.
\end{corollary}

In the next corollary we obtain a splitting of $TT^1S$, invariant under $D\varphi_t$ .

\begin{corollary} $D\varphi_t$ has an invariant singular hyperbolic splitting 
$$
T\left(T^1 S\right) = E^u_\varphi \oplus \RR\dot\varphi \oplus E^s_\varphi,
$$
with $E^u_\varphi $ and $E^s_\varphi $  given by intersecting $E^u_\psi \oplus  \RR\dot\psi$ and $E^s_\psi \oplus  \RR\dot\psi$ with the smooth, $D\varphi_t$-invariant bundle $\dot\varphi^\perp$.
\end{corollary}

Since the weak stable and unstable distributions of a $C^3$ Anosov flow in dimension $3$ are $C^{1+\alpha}$, for some $\alpha>0$, we also obtain:

\begin{corollary} The distributions $E^u_\psi \oplus  \RR\dot\psi$ and $E^s_\psi \oplus  \RR\dot\psi$ 
are $C^{1+\alpha}$, for some $\alpha>0$.  The distributions $E^u_\varphi$ and $E^s_\varphi$ are also $C^{1+\alpha}$, when measured in the $\star$ metric.   Thus in the compact part $\bar\delta\geq \delta_0$, the
distributions $E^u_\varphi$ and $E^s_\varphi$ are uniformly $C^{1+\alpha}$.
\end{corollary}

\begin{remark} If $\bar\delta(p)\leq \delta_1/2$, then the vector $V(p) = \nabla\bar\delta(p)$ points directly away from the cusp.
It is not difficult to see that 
the unstable manifold $\cW^u_\psi(V(p))$ consists of the restriction of the vector field $V$ to the circle
$\bar\delta = \bar\delta(v)$.  For these vectors,  the unstable bundles $E^u_\psi$ and $E^u_\varphi$ coincide.
\end{remark}

Finally we obtain the key bounds on distortion for the flow $\varphi_t$ that will be used to prove exponential mixing.

\begin{corollary}[Distortion control]\label{c=distortion}   For $t\in \RR$ and $v\in T^1S$ denote by  $\|D^s_v\psi_t\|_\star$ and $\|D^u_v\psi_t\|_\star$ the $\star$- norm of the restriction of $D_v\psi_t$ to $E^s_\psi$ and $E^u_\psi$, respectively.  Similarly define
$\|D^s_v\varphi_t\|_\star$ and $\|D^u_v\varphi_t\|_\star$ using the bundles $E^s_\varphi$ and $E^u_\varphi$.
There exist $\theta>0$, $C\geq 1$ and   $\sigma>0$ such that for every $v\in T^1S$:
\begin{enumerate}
\item If $w\in \cW^s_\psi(v, \sigma)$ and $w'\in \cW^u_\psi(v, \sigma)$,  then for all $t>0$:
\[ \left|\log  \|D^s_v\psi_t\|_\star - \log  \|D^s_w\psi_t\|_\star  \right|   \leq  C \rho_\star(v,w)^{\theta},
\]
and
\[ \left |\log  \|D^u_v\psi_{-t}\|_\star - \log  \|D^u_{w'}\psi_{-t}\|_\star  \right|   \leq  C \rho_\star(v,w')^{\theta}.
\]
\item If $w\in \cW^s_\varphi(v, \sigma)$ and $w'\in \cW^u_\varphi(v, \sigma)$,  then for all $t>0$:
\[\left |\log  \|D^s_v\varphi_t\|_\star - \log  \|D^s_w\varphi_t\|_\star  \right |   \leq  C \rho_\star(v,w)^{\theta},
\]
and
\[ \left |\log  \|D^u_v\varphi_{-t}\|_\star - \log  \|D^u_{w'}\varphi_{-t}\|_\star  \right|   \leq  C \rho_\star(v,w')^{\theta}.
\]

\end{enumerate}
\end{corollary}

\begin{proof} The results for $\psi_t$ are standard properties of Anosov flows.  For $\varphi_t$, we need only note that the map induced by $\varphi_t$ between any two $\cW^s_{\varphi}$ manifolds on the same orbit is just the composition of the map induced by $\psi_t$ between the corresponding  $\cW^s_{\psi}$ manifolds with projections along flow lines at both ends between the $\cW^s_{\varphi}$ and $\cW^s_{\psi}$ manifolds.  These latter projections are uniformly $C^2$.
\end{proof}

\begin{remark}  It is not hard to see that the stable and unstable
bundles $E^u_\psi$ and $E^s_\psi$ are not jointly integrable.  It follows that the Anosov flow $\psi_t$ is mixing with respect to the measure $\mu$. This leads to the question: is $\psi_t$ exponentially mixing (if, for example, $\delta_0$ is chosen small enough in the construction)?
\end{remark}

\subsection{Proof of Theorem~\ref{t=accelanosov}}

By standard arguments in smooth dynamics, to prove that $\psi_t$ is an Anosov flow, it suffices to find nontrivial cone fields $\cC^+$ and $\cC^-$ over $T^1 S$ and constants $C>0$, $\lambda>1$
with the properties:
\begin{itemize}
\item $\cC^+(v)\cap \cC^-(v) = \{0\}$, and $\cC^\pm(v)\cap \RR\dot\psi = \{0\}$;
\item $D_v\psi_1(\cC^+(v)) \subset \cC^+(\psi_1(v))$, and $D_v\psi_{-1}(\cC^-(v)) \subset \cC^-(\psi_{-1}(v))$; and
\item For all $t>0$, and all $\xi^+\in \cC^+(v)$ and  $\xi^-\in \cC^-(v)$, we have
$$
\|D_v\psi_t(\xi^+)\|_\star \geq C\lambda^{t}\, \hbox{ and } \|D_v\psi_{-t}(\xi^-)\|_\star \geq C\lambda^{t}.
$$
\end{itemize}

The derivative of $\psi_t$ restricted to $\dot\varphi^\perp$ has a component in the $\dot\varphi$ direction of $TT^1 S$ owing to the time change.  We have:
\begin{equation}\label{e=dpsi}
D_v\psi_t(\xi) = D_v\varphi_{\tau(v,t)} (\xi) =  D_v\varphi_{s}\vert_{s=\tau(v,t)} (\xi) +  D_v\tau(v,t)(\xi) \, \dot\varphi(v).
\end{equation}
 Our strategy to find the cone fields is summarized in two steps.
\begin{enumerate}
\item Use the properties of $D_v\varphi_{\tau_t}$ previously obtained in Lemma~\ref{l=cone2} to define the perpendicular components (i.e. in $\dot\varphi^\perp$) of $\cC^\pm$. 
\item Using a bound on the ``shear term" $D_v\tau_t(\xi)$ in the $\star$ norm,  we then define the components of $\cC^{\pm}$ in the
$\RR\dot\varphi$ direction.
\end{enumerate}
We carry out these steps in the following sections.

\subsubsection{Action of $D\varphi_s$}
Here we fix $t>0$ and study the action of $D_v\varphi_s\colon \dot\varphi^\perp(v) \to \dot\varphi^\perp(\varphi_s(v))$
at $s = \tau(v,t)$. The derivatives of $\tau$ do not enter into these calculations; we are essentially
establishing properties of the original flow $\varphi_s$ (as measured in the $\star$-metric).

\begin{proposition}\label{p=yzstuff}  For any $v\in T^1 S$, and any real numbers $y_0, z_0$
satisfying $z_0/y_0 \in [\chi(v), r+\epsilon]$, the following holds.
For $s\geq 0$, define $y_s$ and $z_s$ by:
$$
D_v\varphi_{s} \left(y_0\bar\delta(v)  Jv, z_0 J v \right) = \left(y_{s}\bar\delta(\gamma_v(s))  J \dot\gamma_v(s), z_{s} J \dot\gamma_v(s)\right).
$$
Then:
\begin{enumerate}
\item $z_s/y_s \in [\chi(\varphi_s(v)), r+\epsilon]$, for all $s\geq 0 $,
and 
\item  for every $t>0$:
 $$\frac{y_{\tau(v,t)}}{y_0}  \geq e^{\beta t}.$$
\end{enumerate}
\end{proposition}

\begin{proof}
For $s>0$, let  $j(s)= \bar\delta(\gamma_v(s)) y_s$.  Then, since 
$(j(s)  J \dot\gamma_v(s), j'(s)  J \dot\gamma_v(s))$
is a perpendicular Jacobi field, the definition of $y_s$, $z_s$ implies that  $j'(s) = z_s$.  In particular,
$j'(s)/j(s)  = z_s/\bar\delta(\gamma_v(s)) y_s$.

Suppose that $z_0/y_0  \in [\chi(v), r+\epsilon]$.   Then
\begin{equation}\label{e=yzcone}
j'(s)/j(s) \in  \left[\frac{\chi(\gamma(s))}{\bar\delta(\gamma_v(s))}, \frac{r+\epsilon}{\bar\delta(\gamma_v(s))}\right]
\end{equation}
holds for $s=0$, and 
Proposition~\ref{p=cones4}  implies that (\ref{e=yzcone}) holds for all $s>0$.  We conclude that
$z_s /y_s \in [\chi(\varphi_s(v)), r+\epsilon]$, for all $s>0$.

Turning to the second item in the proposition, we have that
$$
\frac{\bar\delta(\gamma_v(s)) y_s}{\bar\delta(\gamma_v(0)) y_0 }= \frac{j(s)}{j(0)} =  \exp\left( \int_0^s \frac{j'(u)}{j(u)}\,du\right),
$$
which gives that
$y_s/y_0 = \bar\delta(\gamma_v(0))/\bar\delta(\gamma_v(s)) \exp(  \int_0^s  z_u/(\bar\delta(\gamma_v(u)) y_u) ds),$
and so
\[
y_s  = y_0\exp\left(\int_0^s  \frac{-D\bar\delta(\dot\varphi(\varphi_u(v)))}{\bar \delta(\varphi_u(v))}  + \frac{z_u}{y_u \bar\delta(\varphi_u(v))} \,du \right).
\]
Since $z_u/y_u  \geq  \chi(\varphi_u(v))$, for $u\leq s$, and $\chi - \|\nabla\delta\| >\beta$,  we have
\[
\frac{y_s}{y_0}  \geq \exp\left(\int_0^s \frac{\beta}{\bar\delta(\varphi_u(v))}\,du\right).
\]
We make the substitution $s=\tau(v,t)$  and use the fact that
 $\int_0^{\tau(v,t)} \bar\delta(\varphi_u(v))^{-1} \,du = t$ to obtain the conclusion.
\end{proof}

\subsubsection{Invariant cone fields}

We define invariant stable and unstable cones $\cC^-$ and $\cC^+$.  The angle between  $\cC^+$ and $\cC^-$  will be uniformly bounded in the  $\star$-metric, as will be the angle between either of them and $\dot\psi$.   We establish the properties of  $\cC^+$ in detail; the analogous properties for $\cC^-$ are obtained by the same proof, reversing the direction of time.

Fix $B>0$ to be specified later, and let
\[\cC^+(v):= \{ \left(\bar\delta(v) \left(x v + y Jv\right), z Jv\right) :  z/y \in \left[\chi(v),r+\epsilon\right] \, \& \, |x| \leq B |y|\}\cup\{0\},
\]
and 
\[\cC^-(v):= \{ \left(\bar\delta(v) \left(x v + y Jv\right), z Jv\right) :  z /y \in\left[
-(r+\epsilon), -\chi(v)\right] \,\&\, |x| \leq B |y|\}\cup\{0\}.
\]
Note that since $\chi$ is bounded below away from $0$,  if $\xi =  \left(\bar\delta(v) \left(x v + y Jv\right), z Jv\right) \in \cC^\pm(v)$, then $\|\xi\|_\star$ is uniformly comparable to both $|y|$ and $|z|$.

\begin{lemma} If $B>0$ is sufficiently large, then $D_v\psi_1(\cC^+(v)) \subset \cC^+(\psi_1(v))$, and there exists  $C>0$ such that
$$\xi\in \cC^+(v)  \implies \|D\psi_t (\xi)\|_\star \geq C e^{\beta t} \|\xi\|_\star,
$$
for all $t\geq 0$.
\end{lemma}

\begin{proof} Recall that $\dot\psi(v) = (\bar\delta(v) v,0)$, and
$D_v\psi_t(\dot\psi(v)) = \dot\psi(\psi_t(v))$.  
Applying $D_v\psi_t$ to $\xi = \left(\bar\delta(v) \left(x_0 v + y_0 Jv\right), z_0 Jv\right)$ and using (\ref{e=dpsi}), we get
$$D_v\psi_t(\xi)
=   D_v\psi_t(x_0 \dot\psi(v)) + D_v\psi_t(\bar\delta(v) y_0 Jv, z_0 Jv)$$
\[
=   D_v\psi_t(x_0 \dot\psi(v)) + D_v\varphi_{s}\vert_{s=\tau(v,t)}  \left(y_0\bar\delta(v)  Jv, z_0 Jv\right) + D_v\tau(v,t) \left(y_0\bar\delta(v) Jv, z_0 Jv\right)\dot\varphi(\psi_t(v))
\]
$$ =  \left(x_0 \bar\delta(\psi_t(v)) +D_v\tau(v,t)\left(y_0\bar\delta(v) Jv, z_0 Jv\right)\right)  \left(\psi_t(v),0\right)\qquad\qquad\qquad\qquad\qquad\qquad\qquad\,\,\,
$$
$$\qquad\qquad\qquad\qquad\qquad\qquad\qquad+ D_v\varphi_{s}\mid_{s=\tau(v,t)}\left(y_0\bar\delta(v)  J\psi_t(v), z_ 0 J\psi_t(v)\right).
$$
$$
=: (x_\tau \bar\delta(\varphi_\tau(v)) \varphi_\tau(v), 0 ) + ( y_\tau \bar\delta(\varphi_\tau(v)) J\varphi_\tau(v),  z_\tau J\varphi_\tau(v)),
$$
where in the last expression we've used the abbreviation $\tau = \tau(v,t)$.  

Assume that $\xi\in \cC^+(v)$ and 
without loss of generality that  $y_0\geq 0$.  This implies that
$z_0  \in [\chi(v) y_0, (r+\epsilon)y_0]$  and $|x_0| \leq B |y_0|$. 
 Proposition~\ref{p=yzstuff} implies that  for any $t>0$:
\[|x_\tau|  = |x_0 + \frac{1}{\bar\delta(\varphi_{\tau}(v))} D_v\tau(v,t)\left(y_0\bar\delta(v) Jv, z_0 Jv\right) |  \leq  |x_0| +  \|D_v\tau(t,\cdot)\|_\star \| (y_0\bar\delta(v) Jv, z_0 Jv) \|_\star\]
 \[\leq  |x_0| +   (1+1/\beta)  \|D_v\tau(t,\cdot)\|_\star  |z_0|,\]
$y_\tau \geq e^{\beta t} y_0$, 
and
$z_\tau \in  \left[ \chi(\varphi_\tau(v))  y_\tau, (r+\epsilon) y_\tau\right]$.

Now fix $t=1$, and let $\tau_1 = \tau(v,1)$.
We want to show that $D_v\psi_1(\xi) \in \cC^+(\psi_1(v)) =\cC^+(\varphi_{\tau_1}(v)) $;  i.e. that
$z_{\tau_1}/y_{\tau_1} \in [\chi(\psi_1(v)),r+\epsilon]$  and $|x_{\tau_1}| \leq B |y_{\tau_1}|$.

From the previous discussion, we have $z_{\tau_1} \in  \left[ \chi(\varphi_{\tau_1}(v))  y_{\tau_1}, (r+\epsilon) y_{\tau_1}\right]$,
and setting $C_1 =  (1+1/\beta)\|D_v\tau(1,\cdot)\|_\star$, we also have:
\[|x_{\tau_1}| \leq |x_0| +  C_1|z_0|  \leq B y_0+  C_1|z_0| 
\leq  (B + C_1(r+\epsilon)) y_0 \leq  (B + C_1e(r+\epsilon))e^{-\beta} y_{\tau_1}.
\]
Thus we want choose $B$ such that $(B + C_1(r+\epsilon))e^{-\beta}\leq B$,
which holds if
\[
B\geq \frac{C_1 (r+\epsilon) e^{-\beta}}{1- e^{-\beta}}.
\]
A similar argument works for $y_0<0$.

Finally if $\xi\in \cC^+(v)$, then $\|D\psi_t (\xi)\|_\star$ is uniformly comparable to $|y_\tau|$;
since $|y_\tau|$ grows exponentially on the order $e^{\beta t}$, so does
$\|D_v\psi_t\xi\|_\star$.
\end{proof}

\section{Exponential Mixing}\label{s=expmix}

As mentioned in the introduction, to prove exponential mixing of $\varphi_t$, we will construct a 
Young tower -- a special section to the flow -- whose return times have exponential tails.
Since orbits of $\varphi_t$ spend only a bounded amount of time in the cuspidal region,
ensuring exponential tails for the return time is not difficult.

Fix $\delta_2\leq\delta_1/2$ sufficiently small, and denote by $ Z =  Z(\delta_2)$ the circle $\{p: \bar\delta(p)=\delta_2\}$.
Then $ Z$ lifts to two distinguished circles $ Z^u,  Z^s\subset T^1_{ Z} S$
in the unit tangent bundle:
\[
 Z^u:= \{ \nabla \bar\delta(p) :  p\in  Z\},\; \hbox{ and }  Z^s:= \{- \nabla \bar\delta(p) :  p\in  Z\}.
\]
Then $ Z^u$ is a closed leaf of the unstable foliation $\cW^u$  for $\varphi_t$, and  $ Z^s$ is a closed leaf of the stable foliation.  In a neighborhood $U$ of $ Z^u$ in $T^1S$, there is a well-defined projection
$\pi^{cs}\colon U\to  Z^u$ along local leaves of the weak-stable foliation $\cW^{cs}$ for $\varphi_t$.  Since the foliation
$\cW^{cs}$ is $C^{1+\alpha}$, the map $\pi^{cs}$ is a 
$C^{1+\alpha}$ fibration.

We will prove:
\begin{theorem}\label{t=tower}  For any $v_0\in  Z^u$, there are  constants $C \geq 1$, $\lambda, \alpha \in(0,1)$,  a collection of disjoint, open subintervals $\cI = \{\Delta_j : j\geq 1 \}$, with  $\Delta_j\subset \Delta_0 :=  Z^u\setminus \{v_0\}$, for $j\geq 1$,
and a  function $R\colon \bigcup \cI \to [C^{-1},\infty)$ such that:

\begin{enumerate}
\item  $\left|\Delta_0 \setminus \bigcup \cI \right|=0$, where $|\cdot|$ denotes Lebesgue measure on unstable leaves.
\item  For each $v\in \bigcup \cI$, there exists $v'\in \Delta_0$ such that
$\varphi_{R(v)}(v) \in \cW^s_{loc}(v')$.
\item Define $F\colon \cI \to \Delta_0$ by $F(v) =  \pi^{cs}\varphi_{R(v)}(v)$.
For each $j\geq 1$ there is a  diffeomorphism $h_j:  \Delta_0\to \Delta_j$  such that for all $v\in \Delta_0$:
 \[F\circ h_j(v) = v.\]
\item $h_j$ is a uniform contraction: $d(h_j(v_1), h_j(v_2)) \leq \lambda$.
\item $\log h_j'$ is uniformly $C^{1+\alpha}$:
\[|\log h_j'(v_1)- \log h_j'(v_2)| \leq  C d(v_1,v_2)^{\alpha},
\]
for all $v_1, v_2 \in \Delta_0$.
\item  $\|(R\circ h_j)'\|_\infty \leq C$ for all $j$.
\item   For each $k>0$, we have
$\left| \left\{v\in \bigcup \cI \ :  R(v)\geq k \right\} \right| \leq C  \lambda^{k}$; moreover,
there exists $\epsilon>0$ such that
\[\sum_j \exp(\epsilon |R\circ h_j|_\infty) |h_j'|_\infty<\infty.\]
\item (UNI Condition) For $n\geq 1$, let $R_n = \sum_{i=0}^{n-1} R\circ F^i$ (where defined)
 and let 
\[\cH_n = \{h_{\mathbf j}:= h_{j_n}\circ h_{j_{n-1}} \circ\cdots \circ h_{j_1}: {\mathbf j} = (j_1, \ldots, j_n),\,  j_k\geq 1 \}\] be the set of inverse branches of $F^n$, which satisfy $F^n\circ h_{\mathbf j} =  id_{\Delta_0}$, for all $h_{\mathbf j} \in \cH_n$.  Then there exists $D>0$ such that, for all $N \geq 1$, there exist $n\geq N$ and 
$h_{\mathbf j_1} , h_{\mathbf j_2} \in \cH_n$ such that
\[\inf_{v\in \Delta_0} \left|\left(R_n\circ h_{\mathbf j_1} - R_n\circ h_{\mathbf j_2}\right)'(v) \right| \geq D.
\]

\end{enumerate}

\end{theorem}

A recent result of Ara\'ujo-Melbourne \cite{AM} shows that conditions (1)--(8) imply exponential mixing of $\varphi_t$.
For $\theta\in (0,1]$, define $C^\theta (T^1S)$ to be the set of of $L^\infty$ functions $u \colon T^1S\to \RR$ such that
 $\|u\|_\theta  := |u|_\infty + |u|_\theta < \infty$, where
\[|u|_\theta := \sup_{v \neq v' } \frac{ |u(v)-u(v')|}{\rho(v,v')^\theta}.\]

\begin{corollary}\label{c=main} The flow $\varphi_t$ is exponentially mixing:  for every $\theta\in (0,1]$, there exist constants  $c, C > 0$ such that
for every $u_1, u_2 \in C^\theta(T^1S)$, we have
\[\left| \int_{T^1S} u_1 \, u_2\circ\varphi_t  \,d\vol - \int  u_1\, d\vol  \int  u_2\, d\vol   \right| \leq Ce^{-ct}\|u_1\|_{\theta} \|u_2\|_{\theta},
\]
for all $t>0$.
\end{corollary}

\begin{proof}  In the language of \cite{AM}, conditions (1)-(8) in Theorem~\ref{t=tower} imply that we can express the ergodic flow $\varphi_t$ as the natural extension of 
a $C^{1+\alpha}$ skew product flow satisfying the UNI condition.  See the discussion in \cite{AM} after Remark 4.1. Theorem 3.3 in \cite{AM} then applies to give that $\varphi_t$ is exponentially mixing for a suitable function space of observables, in particular those that are $C^3$.  A standard mollification argument gives exponential mixing for observables
in $C^\theta$ (see Remark 3.4 in \cite{AM}).
\end{proof}

The construction is carried out in two parts.  First, in Subsection~\ref{ss=sections}, we isolate those orbits that leave the thick part of $T^1S$ and travel deeply into the cusp.  These orbits are easily described on a topological level using the Quasi-Clairaut relation developed in Section~\ref{s=QC}.  We give a precise description of the first return map to the thick part for these orbits.  Next, in Subsection~\ref{ss=thick},  we analyze the orbits beginning in $ Z^u$ and decompose into pieces visiting the thin part in a controlled way.  We combine these analyses to obtain the desired decomposition of $\Delta_0$ in Theorem~\ref{t=tower}.

\subsection{Constructing sections to the flow in the cusp} \label{ss=sections}

We will work with $\delta_2 \leq \delta_1/2$ so that for $\bar\delta(p)\leq \delta_2$, we have $\bar\delta(p) = \delta(p)$ and $\chi(p) = r-1-\epsilon$, where $\chi$ is the function appearing in Proposition~\ref{p=cones4}.

Let $ \cT(\delta_2)\subset T^1S$   be the torus  consisting of all unit tangent vectors to $S$ with footpoint in $ Z(\delta_2)$:
\[\cT(\delta_2)= T^1_{ Z(\delta_2)}S.\]
This torus is transverse to the vector field $\dot\varphi$, except at the two circles
\[ C^{+} :=  \{ J\nabla \bar\delta(p) :  p\in  Z\},\; \hbox{ and }  C^{-}:= \{-J \nabla \bar\delta(p) :  p\in  Z\}.
\]

Let $\widehat\cW^u$ and $\widehat\cW^s$ be the laminations of $\cT(\delta_2)$ obtained by
intersecting leaves of the weak foliations $\cW^{cu}$ and $\cW^{cs}$ with $\cT(\delta_2)$.
On $\cT(\delta_2)\setminus ( C^+\cup  C^-)$, the laminations $\widehat\cW^u$ and $\widehat\cW^s$ 
are transverse foliations with $1$-dimensional leaves.
Each lamination $\widehat\cW^u$  and  $\widehat\cW^s$ has exactly one closed leaf,   the curves $ Z^u$ and  $ Z^s$ respectively, which are also unstable and stable manifolds for  $\varphi_t$.

\begin{figure}[ht]
\begin{center}\includegraphics[scale=0.4]{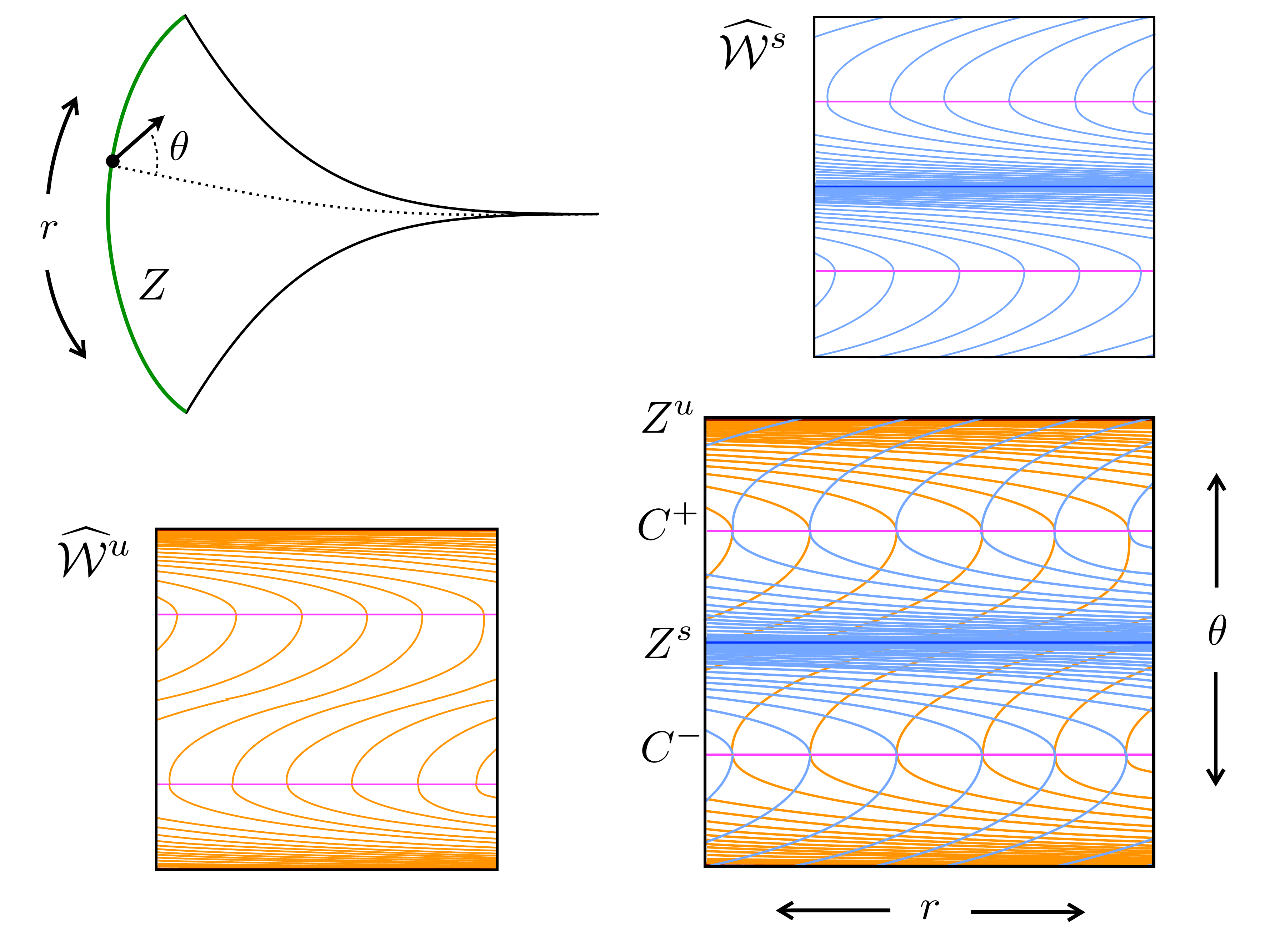}\end{center}
\caption{The laminations  $\widehat\cW^u$ and $\widehat\cW^s$. The singular loci $C^{\pm}$ are the labeled circles where the leaves of $\widehat\cW^u$ and $\widehat\cW^s$ become tangent.}
\end{figure}

For $\eta_0>0$ we define two open subsets $ \cT_{in}(\delta_2, \eta_0)$ and
$ \cT_{out}(\delta_2, \eta_0)$ of $\cT(\delta_2)$ as follows:
\[\cT_{in}(\delta_2, \eta_0) := \{v\in \cT(\delta_2) : a(v) < 0\,\,\&\, |b(v)| \leq \eta_0\}, \]
and
\[\cT_{out}(\delta_2, \eta_0) := \{v\in \cT (\delta_2): a(v) >0\,\,\&\, |b(v)| \leq \eta_0\}, \]
where $a$ and $b$ are defined by (\ref{e=abdef}).
Note that $ Z^u\subset   \cT_{out}(\delta_2, \eta_0)$, and
$ Z^s\subset   \cT_{in}(\delta_2, \eta_0)$, for all $\eta_0>0$.

If $\eta_0<1$, then $ \cT_{out}(\delta_2, \eta_0)$ and  $\cT_{in}(\delta_2,\eta_0)$
are disjoint from $C^{\pm}$, and so $\widehat\cW^u$ and  $\widehat\cW^s$
form uniformly transverse foliations in these cylinders. 

Proposition~\ref{p=Clairaut} implies that if $\delta_2$ is sufficiently small, then for all $\eta_0<1/2$,
there is a well-defined first return map
\[\cR\colon \cT_{in}(\delta_2, \eta_0) \setminus  Z^s \to \cT_{out}(\delta_2, 2\eta_0)\]
for the flow $\varphi_t$, with a local inverse
$\cR^{-1} \colon\cT_{out}(\delta_2, \eta_0)\setminus  Z^u \to \cT_{in}(\delta_2, 2\eta_0)$.  
These maps, where defined,  are $C^3$ and preserve the foliations $\widehat\cW^u$  and  $\widehat\cW^s$.

Fix $v_0\in  Z^u$ and recall that $\cW^u(v_0) = \widehat\cW^u(v_0)=  Z^u$.
 Fix $\eta$ small, and let $I_0=\widehat\cW^s(v_0,\eta)$.  The image of $I_0\setminus\{v_0\}$ under $\cR^{-1}$ is the union of two infinite rays spiraling into the unique closed stable manifold $ Z^s$ in $\cT(\delta_2)$.  Fix another point $v_0' \in  Z^s$ (for example, $v_0' = -v_0$), and fix two points
$v_L',  v_R' \in \widehat\cW^u(v_0',\eta)\cap  \cR^{-1}(I_0)$, to the left and right, respectively,  of $v_0'$  in
 $\widehat\cW^u(v_0',\eta)$ with respect to some fixed orientation.

Let $v_L = \cR(v_L')$, and let $v_R = \cR(v_R')$. The unstable manifold $\widehat\cW^u(v_L)$ contains an infinite ray  from $v_L$, spiraling into $ Z^u$ from the left and cutting $I_0$ infinitely many times.  Let $w_L$ be the first intersection point of this ray with $I_0$; it lies to the right of $v_L$, and to the left of $v_0$.  The points $v_L, w_L$ define a closed
curve $c_L$ in $\cT(\delta_2)$, consisting of the piece of $\widehat\cW^u(v_L)$ connecting 
$v_L$ to $w_L$ and the subinterval of $I_0 =\widehat\cW^s(v_0,\eta)$ from $v_L$ to $w_L$.

\begin{figure}[ht]
\begin{center}\includegraphics[scale=0.28]{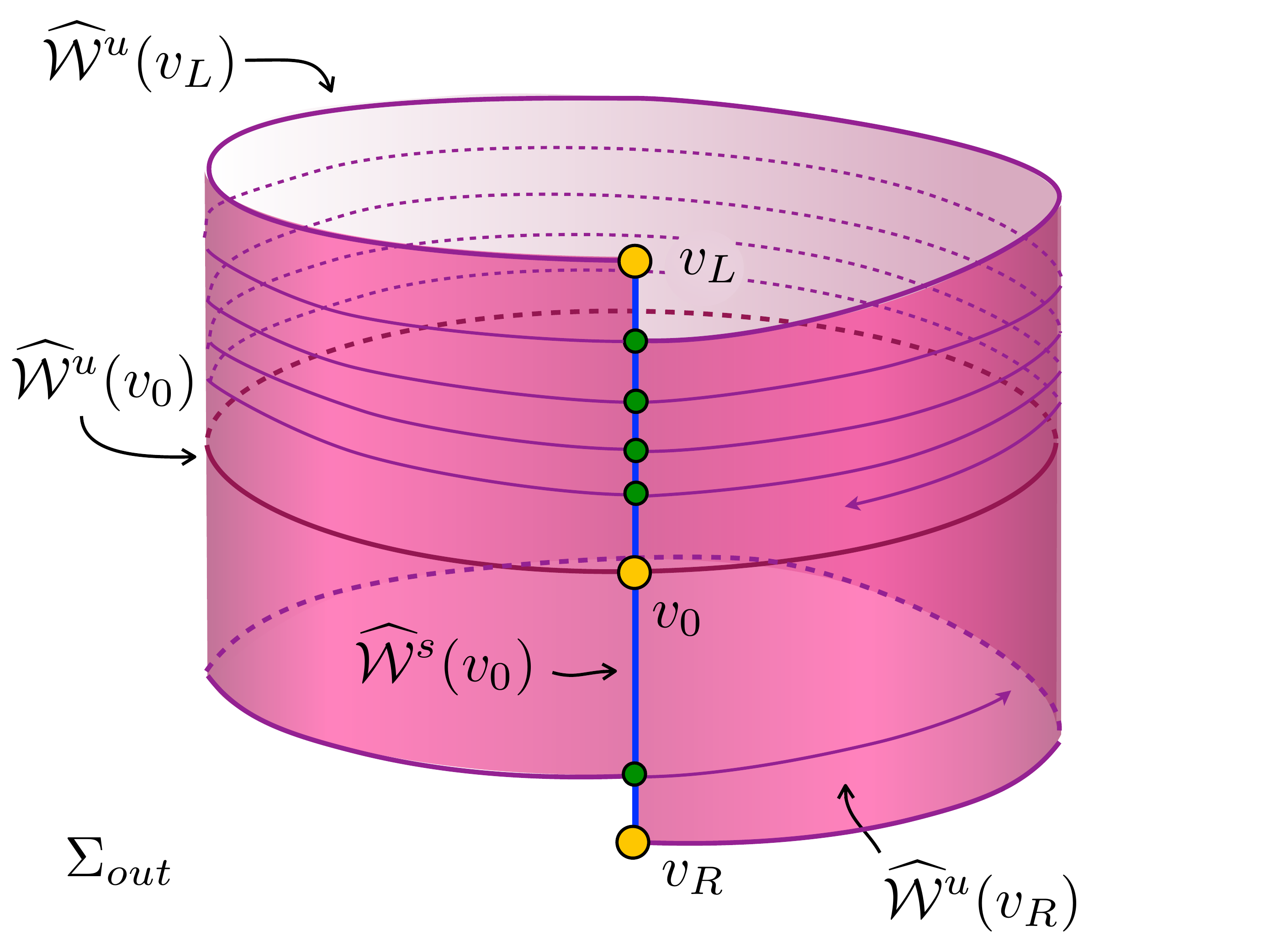}\end{center}
\caption{The section $\Sigma_{out}$ of vectors pointing out of the cusp.}
\end{figure}

  Similarly, let $w_R$ be the first intersection of the 
infinite ray $\widehat\cW^u(v_R)$ spiraling into $ Z^u$ from the right, and let 
$c_R$ be the curve constructed analogously.  The two curves $c_L$ and $c_R$ bound a cylindrical region $\Sigma_{out}$ in $\cT(\delta_2)$, which is depicted in Figure~2.

Let $\pi^s\colon \Sigma_{out}\to  Z^u$ be the projection along leaves of $\widehat\cW^s$, which is simply the
restriction of the projection $\pi^{cs}$ previously defined to the domain $ \Sigma_{out}$.
 Then $\pi^s$ is  $C^{1+\alpha}$ and maps the boundary curves $c_L$ and $c_R$ onto $ Z^u$.  The map $\pi^s$ is a diffeomorphism when restricted to the interior of any interval of $\widehat\cW^s$ that begins and ends in $I_0$ and makes one revolution around $\Sigma_{out}$.

\begin{figure}[ht]
\begin{center}\includegraphics[scale=0.28]{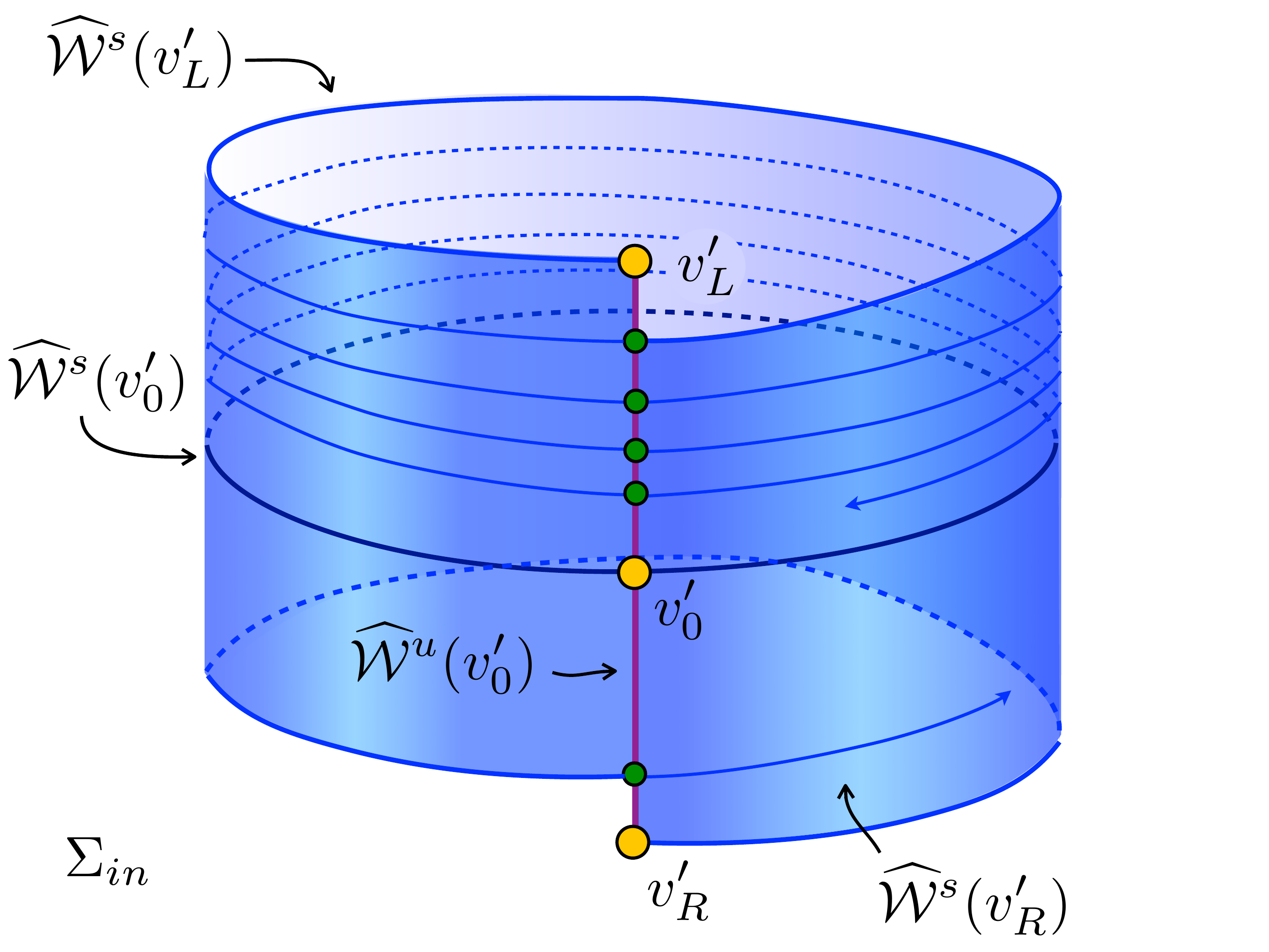}\end{center}
\caption{The section $\Sigma_{in}$ of vectors pointing toward the cusp.}
\end{figure}

We define the section $\Sigma_{in}\subset \cT_{in}(\delta_2,\eta_0)$ similarly: it is bounded by two curves $c_L'$ and $c_R'$, where $c_L'$ is the union of two segments of $\widehat\cW^u(v_L')$ and  $\widehat\cW^s(v_L')$, and $c_R'$ is the union of  segments of $\widehat\cW^u(v_R')$ and  $\widehat\cW^s(v_R')$.  See Figure 3. By construction, we have that $\cR(c_L') = c_L$, $\cR(c_R') = c_R$, and:
\[\cR\left(\Sigma_{in}\setminus  Z^s \right) = \Sigma_{out}\setminus  Z^u.
\]

If the radius $\eta$ of $I_0$ was initially chosen sufficiently small, then there exists $\eta_0\in (0,1/8)$ such that
\begin{equation}\label{e=b0def}
 \cT_{out}(\delta_2, \eta_0/2) \subset \Sigma_{out} \subset \cT_{out}(\delta_2, \eta_0),\;\hbox{ and }
 \cT_{in}(\delta_2, \eta_0/2) \subset \Sigma_{in} \subset \cT_{in}(\delta_2, \eta_0).
\end{equation}
Fix this $\eta_0$.

Let $\pi^u\colon \Sigma_{in}\to  Z^s$ be the projection along leaves of $\widehat\cW^u$, which is
the restriction of the center-unstable $\pi^{cu}$  to $ \Sigma_{in}$.  The fibers of 
$\pi^u$ are pieces of $\widehat\cW^u$-unstable manifold.

Let us examine the return time function for the flow on the fibers of $\pi^u$.  Let $\cN$ be a small neighborhood of $\Sigma_{in}\setminus Z^s$ defined by flowing  $\Sigma_{in}\setminus Z^s$ under
$\varphi_t$ in a small time interval.
 For $v\in \cN$, let
$t_{\cR}(v)$ be the smallest time $t>0$ satisfying $\varphi_t(v)\in \Sigma_{out}$:
\begin{equation}\label{e=defT_1}
t_{\cR}(v) = \inf\left\{t>0 : \varphi_{t}(v)\in \Sigma_{out}\right\};
\end{equation}
thus $\cR(v) = \varphi_{t_{\cR}(v)}(v)$, for all $v\in \Sigma_{in}\setminus Z^s$.

\begin{figure}[ht]
\begin{center}\includegraphics[scale=0.36]{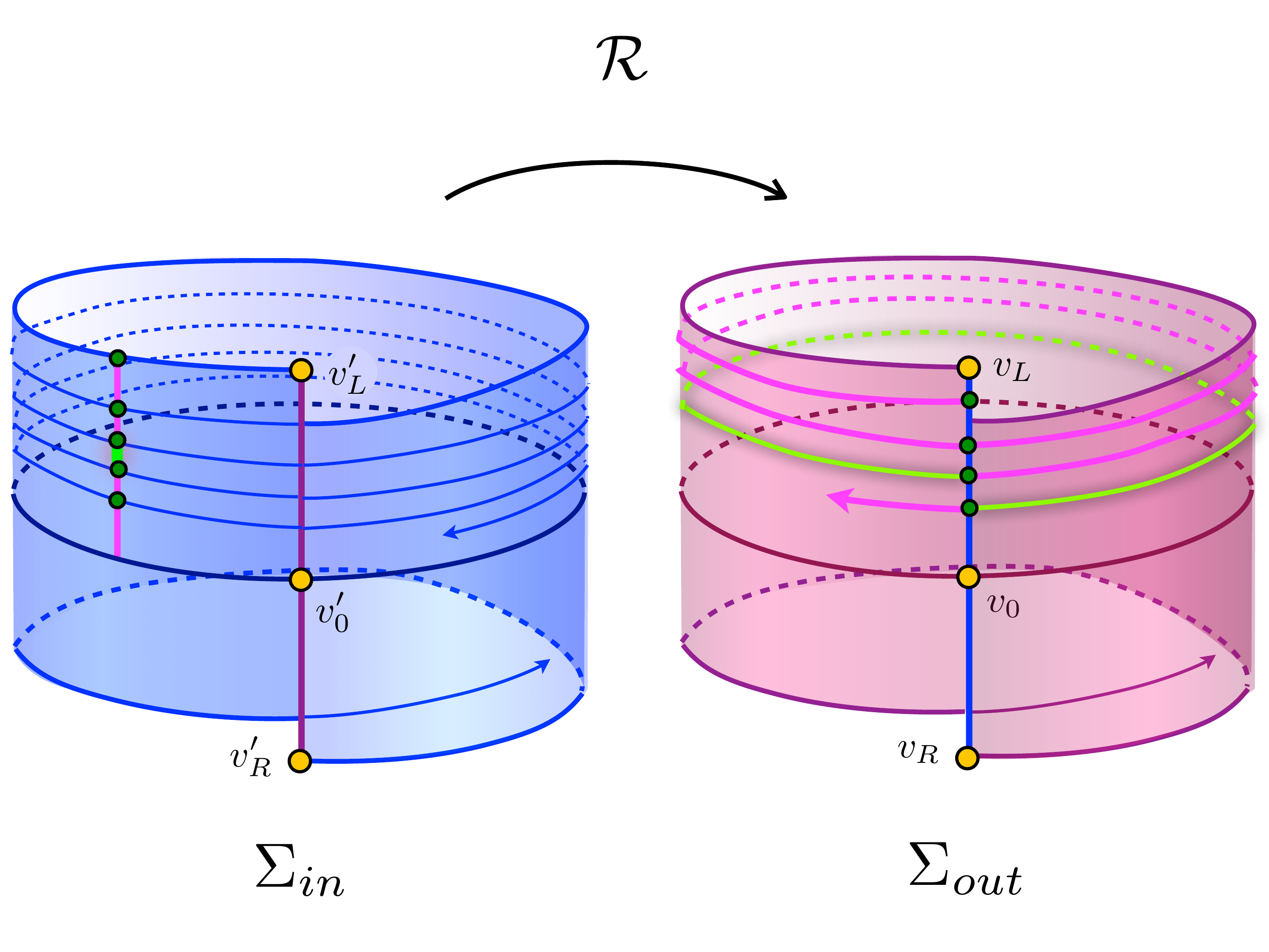}\end{center}
\caption{The action of $\cR$ on fundamental intervals.}
\end{figure}

 Let $v\in  Z^s$ and let $I\subset (\pi^u)^{-1}(v)$ be a closed interval.  We say that $I$ is a {\em fundamental interval} if the endpoints of $I$  lie on the same leaf of the $\widehat\cW^s$ foliation and the interior of $I$ contains no points on that leaf.

\begin{lemma}  There exists  $C_1 \geq 1$  such that for any fundamental interval $I$, the following holds:
\begin{enumerate}
\item $\pi^s\left(  \cR(I) \right) =  Z^u$, and the restriction of $\pi^s\circ\cR$ to the interior of $I$ is a $C^{1+\alpha}$ diffeomorphism, whose inverse has uniformly bounded distortion.
\item For any $w\in I$, we have $\|\cR'(w)\| \asymp |I|^{-1}$, where $\cR'(w)$ denotes the derivative of the restriction of $\cR$ to $I$.
\end{enumerate}
\end{lemma}

\begin{proof} Property (1) follows from the construction of fundamental intervals and the fact that the foliation $\cW^{cs}_{\varphi}$ is uniformly $C^{1+\alpha}$.  Property (2) follows from Corollary~\ref{c=distortion}. \end{proof}

\begin{lemma} \label{l=intervalstretch}
There exist $C_2\geq 1$ and $\alpha>0$ such that if $I$ is a fundamental interval, and
$b(I) = \inf_{v\in I} |b(v)|$, then \[|I| \leq C_2 |b(I)|^{1+\alpha}.\]

\end{lemma}

\begin{proof}  
  For a point $v\in I$, let  $t_1$ be the return time of $v$ for $\varphi_t$  to $\Sigma_{out}$. 
Note that $\delta(\varphi_{t_1}(v))= \delta(v)$ and  $b(\varphi_{t_1}(v)) \approx b(v)$ (by Proposition~\ref{p=Clairaut}).  Let $t_0\in (0,t_1)$ be the point where $a(\varphi_{t_0}(v)) = 0$.

We write $a(t), b(t), c(t)$ for $a(\varphi_t(v)), b(\varphi_t(v)), c(\varphi_t(v))$, respectively.
Lemma~\ref{l=derivs} implies that $c = \frac{1}{|a|} \frac{b'}{b} \geq \frac{b'}{b}$, and
so 
\[\exp\left(\int_0^{t_0} c \right) \geq \exp\left(\int_0^{t_0} \frac{b'}{b} \right) = \exp(-\ln(b(0))) = |b(0)|^{-1}.\]
Thus, since $c = r/\delta + O(1)$ by Proposition~\ref{p=c}, we have
\[\exp\left(\int_0^{t_1} \frac{r}{\delta} \right) \geq C^{-1} |b(0)|^{-2},\]
for some $C\geq 1$,
and
\[ \|\cR'(w)\| \asymp \|D\varphi_{t_1} E^u\| \geq  \exp\left(\int_0^{t_1} \frac{r-1-\epsilon}{\delta} \right) \geq C^{-1} |b(0)|^{-2(r-1-\epsilon)/r}.\]
Then, since distortion is bounded on small intervals by Corollary~\ref{c=distortion}, we have \[|I| \asymp  \|\cR'(w)\|^{-1} \leq C_2 |b(0)|^{2(r-1-\epsilon)/r} \leq C_2 |b(I)|^{1+\alpha},\] for some $\alpha>0$, since $r>2$. \end{proof}

As a corollary, we obtain:
\begin{lemma} \label{l=IBcompare} There exist $C\geq 1$ and $\alpha\in (0,1)$ such that the following holds. If $B$ is a $\widehat W^u$-interval in $\Sigma_{in}$  with $Z^s\cap B\neq \emptyset$, and $I\subset B$ is a fundamental interval, then $|I| \leq C|B|^{1+\alpha}$.  In particular, if $|B|$ is sufficiently small then $|I| \leq |B|/4$.
\end{lemma}

\subsection{Building a Young tower}\label{ss=thick}

\begin{lemma}\label{l=meets} Fix a compact set $K\subset T^1S$, $\eta>0$ and $\sigma>0$. There exists $U_0>0$  such that for any $v\in K$, there exists $w\in \cW^u(v,\sigma)$ and $t\in(0,U_0]$ such that
$\varphi_t(w)\in \cT_{in}(\delta_2,\eta)$.
\end{lemma}
\begin{proof}  This is a consequence of ergodicity (indeed transitivity) of $\varphi_t$,  compactness of  $K$,  and
the Anosov condition on $\psi_t$.

Let $\eta,\sigma$ and $K$ be given. Let $\sigma_1<\eta/8$ be small enough such that for all $v_1,v_2\in K$, if $d(v_1,v_2) <\sigma_1$, then $\cW^s(v_1,\sigma)\cap \cW^{cu}(v_2,\sigma) \neq \emptyset$ (because $\psi_t$ is Anosov, this holds for the restriction of the $\star$ metric to $K$, which is then comparable to the original metric, since $K$ is compact).

 Since $\varphi_t$ is ergodic (by Corollary~\ref{c=ergodic})  there exists $u\in T^1S$ whose backward orbit is dense in $T^1S$ and 
such that \[\varphi_{[-\sigma_1/2,\sigma_1/2]}(u)\cap \cT_{in}(\delta_2,\eta/2) \neq \emptyset.\]
Cover the compact set $K$ with a finite collection of $\sigma_1/2$-balls.
Fix $s_0>0$ such that $\varphi_{-s_0}\left(\cW^s(u,\sigma_1/4))\right)$ has length $>\sigma$.
If $s_1$ is sufficiently large, then $\varphi_{[-s_1,-s_0]}(u)$ meets all of the $\sigma_1/2$-balls, and thus
meets every $\sigma$-center unstable manifold.  This implies the conclusion, with $U_0= s_1+\sigma_1$.
 \end{proof}

We now describe the procedure for partitioning $\Delta_0 =  Z^u\setminus \{v_0\}$ into a full measure set of subintervals $\{\Delta_j: j\geq 1\}$ mapping onto $\Delta_0$ under $\pi^{cs}\circ\varphi_R$, where $R\colon \bigcup_j \Delta_j\to \RR$.

\medskip

  We assume $\delta_2$ and $\eta_0$ are very small, so that by Corollary~\ref{c=distortion} distortion is at most $4/3$ on unstable intervals of length $\leq 2\eta_0$:
\[w\in \cW^u(v,2\eta_0) \implies  \frac{ \|D^u_w\varphi_{-t} \|}{ \|D^u_v\varphi_{-t}\|} \in \left[ \frac34, \frac43 \right], \forall t>0.
\]
Denote by $\ell>0$ the circumference of  $Z^u$, which is less than $1$ if $\delta_2$ is small enough, and
without loss of generality assume $\eta_0 < \ell/2$.
Let $t_{\cR}\colon \Sigma_{in}\setminus Z^s\to \mathbb{R}$ be the return function defined by (\ref{e=defT_1}).

Fix  $\Theta \subset T^1 S$ the thick part defined by
\[\Theta = \{ v\in T^1S : \bar\delta(v) \geq  \delta_2  \}.\]
Note that $\Sigma_{in}\cup\Sigma_{out} \subset \Theta$.
Lemma~\ref{l=meets} implies that there exists $U_0>0$ such that for any $v\in \Theta$,  there exists $w\in \cW^u(v, \eta_0 )$ and $t\in(0,U_0]$ such that $\varphi_{t}(w)\in \cT_{in}(\delta_2, \eta_0 )$.  
Let $s_0$ be the maximum time needed for a piece of unstable interval to double in length under $\varphi_t$.  Let 
$U =  U_0 + s_0 + 2\delta_2$.
Denote by $Z^{cs}$ the singular set consisting of all vectors $v\in T^1S$ with $\bar \delta(v)\leq \delta_2$ and
such that $\varphi_t(v)$ hits the cusp in time $t \leq \delta_2$; that is:
\[Z^{cs}  := \varphi_{[0,\delta_2]}(Z^s).
\]

We begin by chopping $Z^u$ into a collection $\cG_0$ of intervals of length in $[\eta_0 , 2\eta_0 )$.   We say that  an open  interval $G\subset Z^u$ is {\em active gap interval at time $t\geq 0$} if  $\varphi_{[0,t]}(G)\cap Z^{cs}= \emptyset$ and
$|\varphi_t(G)|\in[\eta_0 , 2\eta_0 )$.  Thus $\cG_0$ consists of active gap intervals at time $0$.

Recall from Corollary~\ref{p=leave} that for all $v\in \Sigma_{in}\setminus Z^s$, $\varphi_{2\delta_2}(v)\in \Theta$.
 This implies that if $G\subset Z^u$ is {\em any} piece of unstable manifold of
length less than $\eta_0$ such that:
\begin{itemize}
\item $\varphi_{t_0}(G)\cap \Sigma_{in} \neq\emptyset$, and
\item $\pi^{cs}(\varphi_{t_0}(G))$ contains a fundamental interval,
\end{itemize}
then there exists $t_1\in (0,2\delta_2)$ such that $G$ is an active gap interval at time $t_0+t_1$.

We now describe an algorithm for evolving an active gap interval to produce new active gap intervals and other intervals called border intervals. 

Let $G\subset Z^u$ be an active gap interval at some time $t_0\geq 0$.
 We then flow $G$ forward until the first $t > t_0$ when one of two things happens:
\begin{itemize}
\item[(a)] $|\varphi_t(G)| = 2\eta_0 $, or
\item[(b)] $\varphi_t(G)\cap Z^s\neq \emptyset$.
\end{itemize}
Either (a) or (b) will occur within time $s_0$.  

If (a) happens first, we chop $G$ into two new gap intervals, $G_1$ and $G_2$, so that $|\varphi_{t}(G_1)| = |\varphi_{t}(G_2)| = \eta_0 $.   We say that $G_1$ and $G_2$ 
are {\em born} and become {\em active} at time $t$ and write $t_b(G_1) = t_a(G_1)=t_a(G_2) = t_b(G_2) = t$.  Note that $G = G_1\cup \{v_0\} \cup G_2$, where $v_0$ is the point where the interval $G$ is cut.  Since distortion is bounded by $4/3$ on intervals of length $\leq 2\eta_0 $, we have:
\begin{equation}\label{e=ratio1}
\frac{|G_1|}{|G_2|} \in \left[  \frac34, \frac43 \right],\;\hbox{ and } \frac{|G_i|}{|G|}\in  \left[  \frac38, \frac23 \right], i=1,2.
\end{equation}
We say that $G$ is {\em inactive} in the time interval $[t,\infty)$, and that  $G_1, G_2$ are inactive in the time interval 
$[0,t)$.  

If (b) happens first, then $G$ gives birth to two  {\em border intervals} $B_1$ and $B_2$ and two gap intervals $G_1, G_2$ as follows.  Let $v\in G$ be the unique point satisfying $\varphi_t(v)\in Z^s$. 
Let $X_1$ and $X_2$ be the components of $G\setminus \{v\}$ that lie to the left and right of $v$, respectively.
For $i=1,2$, let $B_i\subset X_i$ be the largest interval satisfying:
\begin{itemize}
\item $\{v\}\cup B_i$ is a closed interval, 
\item $\pi^{cs} \varphi_t(B_i)$ is a countable union of fundamental intervals, and
\item $\pi^{cs} \varphi_t(G_i)$ contains exactly one fundamental interval, where  $G_i = X_i\setminus B_i$.
\end{itemize}

For $i=1,2$, we define the {\em birthday} of $B_i$ and $G_i$ to be $t_b(B_i) =  t_b(G_i) = t$.
Let
$t_a(G_i)$ to be the first time such that $\varphi_{t_a(G_i)}$ is an active interval.  Note that 
since $\varphi_t(G_i)$ meets $\Sigma_{in}$ and $\pi^{cs}\left(\varphi_t(G_i)\right)$ contains a fundamental interval, we have that
$t_a(G_i) \in (t, t + 2\delta_2]$.  Similarly, any fundamental interval in $\pi^{cs}\varphi_t(B_i)$ will return to $\Sigma_{out}$ in time at most $2\delta_2$.

To summarize, in case b) we produce a decomposition of the original active gap interval $G$ into disjoint subintervals 
\[
G = G_1 \cup B_1\cup B_2 \cup G_2,
\]
(up to a finite set of points)
with the following properties:
\begin{itemize}
\item $G_1$ and $G_2$ are gap intervals that are born at time $t_b(G_1) = t_b(G_2) \in [t_0, t_0+s_0]$, respectively.
For $i=1,2$, there exists $t_a(G_i)\in [t_b(G_i), t_b(G_i)+2\delta_2]$ such that $G_i$ is active at time $t_a(G_i)$. We say that
$G_i$ is {\em inactive} in the time period $(0, t_b(G_i))$ and {\em dormant} during the period
$[t_b(G_i), t_a(G_i))$.
\item  $B_1$ and $B_2$ are border intervals that are  born at time $t_b(B_1)=t_b(B_2)\in [t_0, t_0+s_0]$, respectively. For $i=1,2$, the set $\varphi_{t_b (B_i)}(B_i)$ is a countable union of fundamental intervals,
and for any $v\in B_i$, we have $t_{\cR}(\varphi_{t_b(B_i)} (v))\in 
(0, 2\delta_2]$.  
\item Since each $\pi^{cs}\varphi_{t_b(G_i)}(G_i)$ contains exactly one fundamental interval (and no more), it has bounded length when it first returns to $\Theta$:  when this forward image is projected onto $Z^{u}$ it covers at least once, but not more than twice.  Thus, assuming that $\delta_1, \eta_0$ etc. are small enough, we have that
if $t_1 > t_b(G_i)$
is the smallest time such that $\varphi_{t_1}(G_i)\cap \Sigma_{out}\neq \emptyset$,  
then
\begin{equation}\label{e=boundedreturn}
\left|\varphi_{ t_1 }(G_i) \right| \in [\ell/2, 3\ell).
\end{equation}
\item Since  each $\pi^{cs}\varphi_{t_b(G_i)}(G_i)$ is contained in two fundamental intervals, 
Lemma~\ref{l=IBcompare} implies that $|\varphi_{t_b(G_1)} (G_1\cup G_2)| \leq \frac12 |\varphi_{t_b(G_1)} (G)|$;  since distortion is bounded
by $4/3$ on intervals of length $\leq 2\eta_0 $, we have:
\begin{equation}\label{e=ratio2}
|G_1\cup G_2| \leq \frac23 |G|.
\end{equation}
\end{itemize}

Starting with the intervals in $\cG_0$ and applying the algorithm to all active gap intervals, we obtain for any time $t\geq 0$, three disjoint collections of disjoint intervals $\cA_t$, $\cB_t$ and $\cD_t$, the active, border and dormant intervals.   The set $\cA_t$ consists of the gap intervals that are active at time $t$, the set $\cB_t$
consists of border intervals $B$ with birth time $t_b(B) \leq t$, and $\cD_t$ are the gap intervals that are dormant at time $t$.    Let $\cG_t = \cA_t\cup \cD_t$ be the collection of all gap intervals active or dormant at time $t$.

Observe that for any $t>0$, we have
\[Z^u = \bigcup \cB_t \cup  \bigcup \cG_t \cup \bigcup \cV_t, 
\]
where $\cV_t$ is a finite collection of points.  (Note that at $t=0$, we have $\cG_0 = \cA_0$, $\cB_0 =  \cD_0 = \emptyset$,
and $\cV_0 = \{v_0\}$).

\begin{lemma}\label{l=lambda0} There exists $\lambda_0\in(0,1)$ such that for any $k\geq 0$:
\[\left| Z^u\setminus \bigcup \cB_{kU} \right|   = \left| \bigcup \cG_{kU} \right|  \leq \lambda_0^k.
\]
\end{lemma}
\begin{proof}  Let $G\in \cG_{kU}$.  Then $G$ is either active or dormant at time $kU$.  Since an interval cannot be active for more than time $s_0\leq U$ and cannot be dormant for more than time $2\delta_2\leq U$, it follows that $G$ is inactive at time $(k-1)U$.  It follows that $G$ has a unique {\em ancestor} in $\cG_{(k-1)U}$; that is, there exists $G'\in \cG_{(k-1)U}$ such that $G\subset G'$ and $G'$ gives birth in the time interval $[(k-1)U, kU]$.

Now suppose $G'\in \cG_{(k-1)U}$.  Then $G'$ will become active within time $2\delta_2$ and some point $v\in G'$ will intersect $Z^s$ within time $2\delta_2 + U_0\leq U$.  
Thus during the time
period $[(k-1)U, kU]$, the interval $G'$ will divide finitely many times, and at least one active piece will intersect $Z^s$.  

The number of times this division can occur is uniformly bounded.  As the gap evolves in the time interval $[(k-1)U, kU]$, it gives birth to new gaps according to rule (a) or (b) above.  The number of times that case (a) can apply between two occurrences of case (b) is bounded: if a gap $G''$ is produced by rule (b), then by (\ref{e=boundedreturn}), we have $|\varphi_{t_1}(G'')|\in [\ell/2, 3\ell)$, where $t_1\geq t_b(G'')$ is the first time $G''$ returns to $\Sigma_{out}$.  In $\Theta$, the derivative $D\varphi_t$ is bounded above,  and so any active interval meeting $\Theta$ can divide a bounded number of times before some descendent meets $Z^s$ (which happens within time $U_0$).  Thus the number of times (a) can apply within two occurrences of (b) is uniformly bounded.

We conclude that $G' = G_1\cup \cdots \cup G_n \cup B_1\cdots \cup B_{2m}$,
with $n\geq 2m\geq 2$, where $G_1,\ldots,G_n\in \cG_{kU}$, and
$B_{1},\ldots, B_{2m}\in \cB_{kU}$.  Moreover, there exists $N>0$, independent of $k, G'$  such that $n\leq Nm$.
Combined with (\ref{e=ratio1}) and (\ref{e=ratio2}), this implies that there exists $\lambda_0\in (0,1)$ such that
\[\left| G_1\cup \cdots  \cup G_{n} \right| \leq \lambda_0 |G'|.
\]
Thus  $|\cG_{kU}|\leq \lambda_0 |\cG_{(k-1)U}|$; since $|\cG_{0}| < 1$, we obtain the conclusion.
\end{proof}

Let $\cB_\infty = \bigcup_{t>0} \cB_t$.  Then $\cB_\infty$ is a collection of disjoint intervals with
$\left|Z^u \setminus \bigcup \cB_\infty  \right| = 0$.

\begin{proof}[Proof of Theorem~\ref{t=tower}]

 We create a countable collection $\cI$ of intervals $\Delta_j$ as follows:
we decompose each $B\in \cB_\infty$ into a countable union $B = \bigsqcup_{j\geq 1} \Delta_{B,j}$ such that for each $j$,
$\pi^{cs}\varphi_{t_b(B)}(\Delta_{B,j})$ is a fundamental interval.  Then we set
\[\cI =  \{ \Delta_{B,j} :\, B\in \cB_\infty, \, j\geq 1\}.
\]
Note that 
$\left| Z^u\setminus \bigcup \cI \right| = \left| Z^u\setminus \bigcup \B_\infty \right| = 0$,
and so conclusion (1) of Theorem~\ref{t=tower} holds.

We extend the definition of $t_b$ to intervals in $\cI$ in the natural way: if $\Delta_j \subset B \in \cB_{\infty}$, we set $t_b(\Delta_j) = t_b(B)$.
For $v\in \Delta_j \subset \cI$, let
\[R_0(v) = t_b(\Delta_j) + t_{\cR}(\varphi_{t_b(\Delta_j)} (v)).\] Then
$R_0(v)$ is the the minimal time $> t_b(\Delta_j)$ such that
 $\varphi_{R_0(v)}(v)\in \Sigma_{out}$.
\begin{lemma}\label{l=R_0bound} There exist $\lambda\in(0,1)$ and $C\geq 1$ such that the function $R_0\colon \bigcup \cI\to \RR_{>0}$ satisfies
\[\left| \left\{v\in \bigcup\cI :  R_0(v)\geq k \right\} \right| \leq C \lambda^k,
\]
for each $k\geq 0$.
\end{lemma}
\begin{proof} Since $t_{\cR}$ is bounded it suffices to find $\lambda_1\in(0,1)$ such that
\[\left| \bigcup  \left\{B\in \cB_\infty :  t_b(B)\geq k \right\} \right| \leq C \lambda_1^k.\] But this follows immediately from the construction with $\lambda_1 = \lambda_0$ appearing in Lemma~\ref{l=lambda0}.
\end{proof}

We now define the return time function $R\colon \bigcup I\to \RR_{>0}$. Recall the projection
$\pi^s\colon \Sigma_{out} \to Z^u$ along the leaves of $\widehat\cW^s$.  The fibers
of $\pi^s$ are local $\widehat\cW^s$ manifolds.  Over each point $v\in \Sigma_{out}$ there
lies a unique point  $w(v)\in \cW^s_{loc}(\pi^s(v))$ such that $w(v) = \varphi_{r(v)}(v)$, for some small value of $r(v)$.  Let $\bar \Sigma_{out} = w(\Sigma_{out})$.  

\begin{lemma}\label{l=barSigma} The function
$r\colon \Sigma_{out}\to \RR$ is  uniformly $C^{1+\alpha}$, and  $\bar\Sigma_{out}$ is a $C^{1+\alpha}$ manifold.   The map $w\colon \Sigma_{out} \to \bar \Sigma_{out}$ is a $C^{1+\alpha}$ diffeomorphism.  
The manifold $\bar \Sigma_{out}$ is $C^{1+\alpha}$ foliated by local $\cW^s$-leaves:
\[\bar\Sigma_{out} \subset \bigcup_{v'\in \Delta_0} \cW^s_{loc}(v'),
\]
and the projection $\bar\pi^s \colon\bar\Sigma_0\to Z^u$ along these local leaves is a $C^{1+\alpha}$ submersion.
\end{lemma}
\begin{proof} This follows from the fact that the foliation $\cW^s$ is uniformly $C^{1+\alpha}$.
\end{proof}

We define $R\colon \bigcup \cI \to \RR$ by $R(v) = R_0(v) + r(\varphi_{R_0(v)}(v))$;
it has the property that $\varphi_{R(v)}(v)\in \bar\Sigma_{out}$. 
Lemma~\ref{l=barSigma} implies that for each $v\in \bigcup \cI$, there exists a unique $v'\in \Delta_0$ -- namely, 
$v' = \bar\pi^s\varphi_{R(v)}(v)$ --  such that
$\varphi_{R(v)}(v)\in \cW^s_{loc}(v')$, giving conclusion (2).

For $\Delta_j\in \cI$, we  define $h_j\colon \Delta_0\to \Delta_j$ to be the inverse of the map $F_j = \pi^{cs}\circ \varphi_{R(\cdot)}  = \bar\pi^s \circ \varphi_{R(\cdot)} \colon \Delta_j\to \Delta_0$.  
This
is well-defined, because 
\[\bar\pi^s(\varphi_{R(\cdot)}(\Delta_j)) = \pi^{s}(\cR( \pi^{cs} \circ\varphi_{t_b(\Delta_j)}(\Delta_j))) = \Delta_0,\]
since $ \pi^{cs} \varphi_{t_b(\Delta_j)}(\Delta_j)$ is a fundamental interval.
Since $\pi^s$ is a  submersion, and  the map $v \mapsto \cR( \pi^{cs} \varphi_{t_b(B)})(v)$ is a  diffeomorphism  from $\Delta_j$ onto its image, the composition is a  diffeomorphism from $\Delta_j$ to $\Delta_0$.  Thus its inverse $h_j\colon \Delta_0\to \Delta_j$ is a diffeomorphism.  This establishes conclusion (3).

Note that $|h_j'(v)| \sim \|D^u_{h_j(v)}\varphi_{R_0(h_j( v))}\|^{-1}$, and so conclusion (4) holds.
Conclusion (5) follows from the facts that  $\varphi_t$ has bounded distortion and the map
$\pi^{cs}$ is uniformly $C^{1+\alpha}$.  Indeed note that the map $F_j = h_j^{-1}$ can also be expressed in the following way.  We fix some point $\hat v\in \Delta_j$ and
consider the image $\varphi_{R(\hat v)}(\Delta_j)$ under the constant time flow $\varphi_{R(\hat v)}$.  This is a piece of unstable manifold that meets $\bar \Sigma_{out}$.  The map $F_j = h_j^{-1}$ is
just the composition $F_j = \pi^{cs}\circ \varphi_{R(\hat v)}$ of this flow with the center-stable projection $\pi^{cs}\colon U \to Z^u$, defined in the beginning of the section. This latter projection is a uniformly $C^{1+\alpha}$ submersion and a local diffeomorphism when restricted to local unstable manifolds,  since the foliation
$\cW^{cs}$ is uniformly $C^{1+\alpha}$.   Thus $h_j$ is uniformly $C^{1+\alpha}$.

 Let's examine the map $R\circ h_j$.  Again fix a point $\hat v\in \Delta_j$,  and consider the image $\varphi_{R_0(\hat v)}\left(\Delta_j \right)$, which is a piece of unstable manifold meeting $\Sigma_{out}$ at the point $\varphi_{R_0(\hat v)}( \hat v)$.   It follows that there is a uniformly bounded $C^{2}$ function $\hat r\colon \varphi_{R_0(\hat v)}\left(\Delta_j \right) \to\RR$ such that 
$\varphi_{\hat r(\cdot)}$ sends the piece of unstable manifold $ \varphi_{R_0(\hat v)}\left(\Delta_j \right)$  to  $\widehat\cW^u_{loc}(\varphi_{R_0(\hat v)}( \hat v))\subset \Sigma_{out}$.   Then $R_0(v) = R_0(\hat v) + \hat r( \varphi_{R_0(\hat v)}( v))$, and so
$R(v) = R_0(\hat v) + \hat r( \varphi_{R_0(\hat v)}( v) )+ r(\varphi_{R_0(v)}(v))$.
Thus
\[|R'(v)| \leq |\hat r'( \varphi_{R_0(\hat v)}( v) )|\|D^u_v\varphi_{R_0(\hat v)}\|  \qquad \qquad \qquad \qquad \qquad \qquad \qquad \qquad \qquad \qquad \]\[ + |r'(\varphi_{R_0(v)}(v))|\left(\|D^u_v\varphi_{R_0(v)}\||+   \|\dot\varphi(\varphi_{R_0(v)}(v)\| |\hat r'( \varphi_{R_0(\hat v)}( v) )| \|D^u_v\varphi_{R_0(\hat v)}\|  \right)\qquad\]
\[=  |\hat r'( \varphi_{R_0(\hat v)}( v) )|\|D^u_v\varphi_{R_0(\hat v)}\|(1+  |r'(\varphi_{R_0(v)}(v))|)  +  |r'(\varphi_{R_0(v)}(v))|\|D^u_v\varphi_{R_0(v)}\||.
\]
The derivatives $r'$ and $\hat r'$ are uniformly bounded.
Since $|h_j'(v)| \asymp \|D^u_{h_j(v)}\varphi_{R_0(h_j( v))}\|^{-1}$, we obtain that
there exists a uniform constant $C\geq 1$ such that
$|(R\circ h_j)'|\leq C$, for all $j$.  This gives conclusion (6).

 Since the function $r$ is bounded, Lemma~\ref{l=R_0bound} implies that for each $k>0$, we have
\[\left| \left\{v\in \bigcup\cI :  R(v)\geq k \right\} \right| \leq C \lambda^k;
\]
since $|h_j'|\asymp |\Delta_j|$, this gives conclusion (7) of Theorem~\ref{t=tower}, where $\epsilon>0$ is chosen so that $\lambda \exp(C\epsilon) < 1$.

Finally we verify that the UNI Condition in conclusion (8) holds. 
This is a direct consequence of the fact that $\varphi_t$ preserves a contact $1$-form $\omega$, which implies that the foliations $\cW^s$ and $\cW^u$ are not jointly integrable.   The details are carried out in  Lemma 12 of \cite{ABV} (in the Axiom A context) and Lemma 4.2 and Corollary 4.3 of \cite{AM} (close to the current context). \end{proof}

\end{document}